\documentclass[12pt, a4paper]{amsart}

\usepackage[matrix,arrow,curve,frame]{xy}
\usepackage{amsmath,amsthm,amssymb,enumerate}
\usepackage{latexsym}
\usepackage{amscd}
\usepackage[backref]{hyperref}
\usepackage{euscript}

\usepackage{fullpage}

\usepackage[all]{xypic}

\newtheorem{theorem}{Theorem}[section]
\newtheorem{lemma}[theorem]{Lemma}
\newtheorem{corollary}[theorem]{Corollary}

\theoremstyle{definition}

\newtheorem{example}[theorem]{Example}

\theoremstyle{remark}
\newtheorem{remark}[theorem]{Remark}

\numberwithin{equation}{section}
\newcommand{\beq}{\begin{equation}}
\newcommand{\eeq}{\end{equation}}

\newcommand{\ZZ}{\mathbb{Z}}

\def\C{{\bf C}}

\def\cU{{\cal U}}
\def\cG{{\cal G}}

\def\cA{{\cal A}}

\def\ra{\rightarrow}
\def\bra{\langle}
\def\ket{\rangle}

\def\cA{{\mathcal A}}
\def\cB{{\mathcal B}}

\def\cE{{\mathcal E}}
\def\cF{{\mathcal F}}
\def\cG{{\mathcal G}}
\def\cH{{\mathcal H}}
\def\cI{{\mathcal I}}
\def\cJ{{\mathcal J}}

\def\cM{{\mathcal M}}

\def\cS{{\mathcal S}}

\def\cU{{\mathcal U}}
\def\cV{{\mathcal V}}
\def\cW{{\mathcal W}}

\def\ga{{\mathfrak a}}
\def\gb{{\mathfrak b}}

\def\gg{{\mathfrak g}}
\def\gh{{\mathfrak h}}

\def\gj{{\mathfrak j}}

\newcommand{\longhookrightarrow}{\ensuremath{\lhook\joinrel\relbar\joinrel\rightarrow}}

\newcommand {\be}{\begin{equation}}
\newcommand {\ee}{\end{equation}}
\newcommand{\h}{\begin{eqnarray*}}
\newcommand{\e}{\end{eqnarray*}}

\begin{document}


\title
{Twisted Chiral De Rham Complex, \\ Generalized Geometry, and T-duality}


\author{Andrew Linshaw}
\address{Department of Mathematics,
University of Denver, Denver, CO 80208, USA}
\email{andrew.linshaw@du.edu}

 \author{Varghese Mathai}
\address{School of Mathematical Sciences,
University of Adelaide, Adelaide 5005, Australia}
\email{mathai.varghese@adelaide.edu.au}

\date{}
\thanks{ AL was partially supported
by the grant 318755 from the Simons Foundation. 
VM was supported by funding from the Australian Research Council, through Discovery Projects
DP110100072 and DP130103924. 
}

\maketitle

\begin{abstract}
The chiral de Rham complex of Malikov, Schechtman, and Vaintrob, is a sheaf of differential graded vertex algebras that exists on any smooth manifold $Z$, and contains the ordinary de Rham complex at weight zero. Given a closed 3-form $H$ on $Z$, we construct the twisted chiral de Rham differential $D_H$, which coincides with the ordinary twisted differential in weight zero. We show that its cohomology vanishes in positive weight and coincides with the ordinary twisted cohomology in weight zero. As a consequence, we propose that in a background flux, Ramond-Ramond fields can be interpreted as $D_H$-closed elements of the chiral de Rham complex. Given a T-dual pair of principal circle bundles $Z, \widehat{Z}$ with fluxes $H, \widehat{H}$, we establish a degree-shifting linear isomorphism between a central quotient of the $i\mathbb{R}[t]$-invariant chiral de Rham complexes of $Z$ and $\widehat{Z}$. At weight zero, it restricts to the usual isomorphism of $S^1$-invariant differential forms, and induces the usual isomorphism in twisted cohomology. This is interpreted as T-duality in type II string theory from a loop space perspective. A key ingredient in defining this isomorphism is the language of Courant algebroids, which clarifies the notion of functoriality of the chiral de Rham complex. \end{abstract}

\section*{Introduction}
It is a folklore principle that the equations of string theory are formulated on the space of free loops $LZ$ where $Z$ is spacetime. Constructions involving the space of free loops lead to vertex algebras, whose relevance to strings in string theory is comparable to that of Lie algebras in the classical physics of point particles. For example, vertex algebras arise naturally in the projective representation theory of loop groups, cf. Chapter 13 in \cite{PS}. 
The notion of vertex algebra also arises naturally from the Wightman axioms for quantum field theory, see Chapter 1 in \cite{K}.
A central example relevant to this paper is the chiral de Rham complex of Malikov, Schechtman and Vaintrob \cite{MSV1,MSV2}, which is a sheaf of vertex algebras $\Omega^{\text{ch}}_Z$ that exists on any smooth manifold $Z$ in either the $C^{\infty}$, holomorphic, or algebraic categories. It was formalised in the algebraic category in \cite{KV}. See for example \cite{Witten07,KO,Borisov} for the tremendous impact of the chiral de Rham complex in string theory. Let $L^+Z$ denote the subspace of $LZ$ consisting of null-homotopic loops on $Z$. There is a projection map $p: L^+Z \to Z$ given by evaluation at $0$. Roughly speaking, the chiral de Rham complex of $Z$ is the image under $p$ of the semi-infinite de Rham complex of
the D-module of Dirac delta functions along $L^+Z$. The algebra of global sections $\Omega^{\text{ch}}(Z)$ has a grading by conformal weight, and the weight zero component is isomorphic to the usual de Rham complex $\Omega^*(Z)$. The de Rham differential $d$ extends to a vertex algebra derivation $D$ on $\Omega^{\text{ch}}(Z)$. Even though $\Omega^{\text{ch}}(Z)$ is much larger, the cohomology $H^*(\Omega^{\text{ch}}(Z), D)$ vanishes in positive weight, so the inclusion $(\Omega^*(Z), d) \hookrightarrow (\Omega^{\text{ch}}(Z), D)$ is a quasi-isomorphism.

This paper is also partly motivated by recent developments in string theory in a background flux. In \cite{BM00}, it was argued that D-brane charges in a background H-flux take values in twisted K-theory of spacetime $Z$, $K^\bullet(Z,H)$. The Chern-Weil representatives of the twisted Chern character
$Ch_H:  K^\bullet(Z,H) \to H^\bullet(Z,H)$ taking values in the twisted de Rham cohomology, Ramond-Ramond fields, were defined and their properties studied in \cite{BCMMS,MS}. Also central is generalized geometry initiated by Hitchin, and developed by him and his students \cite{Hitchin,Hitchin2,gualtieri}, and in particular exact Courant algebroids \cite{Severa1,Severa2}.

Our first goal in this paper is to construct the {\it twisted chiral de Rham differential} $D_H$ on $\Omega^{\text{ch}}(Z)$. Regarding $\Omega^{\text{ch}}(Z)$ as a module over itself rather than a vertex algebra, $D_H$ is a square-zero derivation in the category of modules. It preserves conformal weight and is homogeneous with respect to a $\mathbb{Z}/2\mathbb{Z}$ grading, and we denote this $\mathbb{Z}/2\mathbb{Z}$-graded complex by $(\Omega^{\text{ch},\bullet}(Z), D_H)$. The weight zero subcomplex coincides with the classical twisted de Rham complex $(\Omega^{\bullet}(Z), d_H)$. As in the untwisted setting, the cohomology $H^{\bullet} (\Omega^{\text{ch}}(Z), D_H)$ vanishes in positive weight and coincides with $H^{\bullet}(Z,H)$ at weight zero. We propose that the configuration space of Ramond-Ramond fields in a background H-flux is the subspace of $\Omega^{\text{ch}}(Z)$ that is closed under $D_H$. This perspective could be more natural in the context of string theory since it incorporates massive Ramond-Ramond fields, where the mass is given by the conformal weight. Even though this configuration space is much richer than the $d_H$-closed differential forms and captures information about the loop space of $Z$, the D-brane charges are the same as in the classical setting since the twisted cohomology vanishes in positive weight. We mention in passing that there are also very different types of twists of the chiral de Rham complex, as studied in \cite{GMS1,GMS2}.

\subsection{T-duality} Our second goal in this paper is to establish a version of T-duality in the chiral setting. T-duality for pairs $(Z,H)$ consisting of nontrivial circle bundles, together with degree 3 H-flux with integral periods, was originally studied in detail in \cite{BEM,BEM2,BHM,BHM05}.  In string theory, T-dual pairs are distinct compactification manifolds that cannot be distinguished by any experiment, which implies the isomorphisms of a number of other structures, such as Courant algebroids \cite{cavalcanti}, generalized complex structures \cite{cavalcanti} and twisted K-theory \cite{BEM}, see also \cite{BS,RR88}. More precisely, the following situation is studied.
\begin{equation}
\xymatrix{(Z,H)\ar[dr]_{\pi}&&
(\widehat Z, \widehat H)\ar[dl]^{\widehat \pi}\\
& M&}
\end{equation}
where $Z, \widehat Z$ are principal circle bundles over a base $M$ with fluxes $H$ and $\widehat H$, respectively, satisfying
$\pi_*(H)=c_1(\widehat Z), \, \widehat \pi_*(\widehat H)=c_1(Z)$ and $H-\widehat H$ is exact on the
correspondence space $Z\times_M \widehat Z$. Then Bouwknegt, Evslin and Mathai \cite{BEM,BEM2} proved that
there is an isomorphism of twisted K-theories 
\begin{equation}\label{tduality-kthy}
K^\bullet(Z,H) \cong K^{\bullet+1}(\widehat Z, \widehat H)
\end{equation}
and an isomorphism of twisted cohomology theories, 
\begin{equation}\label{tduality-coh}
H^\bullet(Z,H) \cong H^{\bullet+1}(\widehat Z, \widehat H).
\end{equation}
In fact, \eqref{tduality-coh} is induced by a degree-shifting linear isomorphism of $S^1$-invariant differential complexes \begin{equation} \label{tduality-chain} T: (\Omega^{\bullet}(Z)^{S^1} ,d_H) \stackrel{\cong}{\longrightarrow} (\Omega^{\bullet + 1}(Z)^{S^1}, d_{\tilde{H}}).\end{equation} As shown by Cavalcanti and Gualtieri in \cite{cavalcanti}, there is also an isomorphism of Courant algebroids \begin{equation} \label{tduality-CG} \tau: (TZ \oplus T^*Z, [\cdot,\cdot]_H)/S^1 \stackrel{\cong}{\longrightarrow} (T\widehat{Z} \oplus T^*\widehat{Z}, [\cdot,\cdot]_{\widehat{H}})/S^1,\end{equation} which is compatible with \eqref{tduality-chain}, regarded as an isomorphism of the associated Clifford modules.

In the chiral setting, it is immediate from \eqref{tduality-coh} and the positive weight vanishing theorem for twisted chiral de Rham cohomology that we have an isomorphism \begin{equation} \label{intro-chiraltduality-coh} H^{\bullet}(\Omega^{\text{ch}}(Z), D_H) \stackrel{\cong}{\longrightarrow} H^{\bullet + 1}(\Omega^{\text{ch}}(\widehat{Z}), D_{\widehat{H}}).\end{equation} such that the following diagram commutes.


\begin{equation}\label{tduality-commute}
\xymatrix @=4pc 
{ H^\bullet(\Omega^{\text{ch}}(Z), D_H)  \ar[d]_{\cong}  \ar[r]^{} &
H^{\bullet+1}(\Omega^{\text{ch}}(\widehat Z), D_{\widehat H}) \ar[d]^{\cong}   \\ H^\bullet(Z, H)  \ar[r]_{}& H^{\bullet +1}(\widehat Z, \widehat H)}
\end{equation}
This isomorphism is interpreted as {\em T-duality from a loop space perspective}, giving an equivalence (rationally) between 
D-brane charges in a background H-flux in type IIA and IIB string theories. However, this is only a preliminary result because the isomorphism in cohomology is not induced by an isomorphism of the underlying chiral structures.

Our next goal is to establish the chiral analogues of the degree-shifting isomorphism \eqref{tduality-chain} of $S^1$-invariant twisted de Rham complexes, as well as the isomorphism \eqref{tduality-CG} of Courant algebroids. The difficulty is that the chiral de Rham complex is not functorial since its construction involves both differential forms and vector fields. The language of Courant algebroids is needed to establish the right notion of functoriality in this setting. The connection between Courant algebroids and the chiral de Rham complex was first observed by Bressler \cite{Bressler}, and was developed further by Heluani \cite{Heluani1}. An important result of \cite{Heluani1} is that for every Courant algebroid $E$ on $Z$, there is a sheaf $\cU^E$ of vertex algebras on $Z$, which coincides with $\Omega^{\text{ch}}_Z$ when $E$ is the standard Courant algebroid $TZ \oplus T^*Z$. Specializing this to exact Courant algebroids $E = (TZ \oplus T^*Z, [\cdot,\cdot]_H)$, we obtain the {\it $H$-twisted chiral de Rham complex} which we shall denote by $(\Omega^{\text{ch},H}_Z, \tilde{D})$. The assignment
$$\Big\{\text{exact Courant algebroid}\Big\} \Longrightarrow  \Big\{\text{twisted chiral de Rham complex} \Big\}$$ is clearly functorial. Morphisms of exact Courant algebroids have been well studied \cite{Bursztyn,Li-Bland}, and they induce morphisms of twisted chiral de Rham complexes. It is surprising that although exact Courant algebroids are classified by $H^3(Z, \mathbb{R})$, for any choice of $H$ there is an {\it untwisting} isomorphism \begin{equation} \label{intro-untwisting}(\Omega^{\text{ch}}_Z, D) \ra (\Omega^{\text{ch},H}_Z, \tilde{D})\end{equation} of sheaves of differential graded vertex algebras.

We may also regard $\Omega^{\text{ch},H}_Z$ as a sheaf of modules over itself. It has a twisted differential $\tilde{D}_H$ which is a square-zero derivation in the category of modules. The differential preserves weight and is homogeneous with respect to a $\mathbb{Z}/2\mathbb{Z}$ grading, and we denote this complex by $(\Omega^{\text{ch},H,\bullet}_Z, \tilde{D}_H)$. We have an isomorphism of sheaves of differential graded modules \begin{equation} \label{intro-untwistingmodule} (\Omega^{\text{ch},\bullet}_Z, D_H) \ra (\Omega^{\text{ch},H,\bullet}_Z, \tilde{D}_H) .\end{equation}

Next, let $\pi: Z \rightarrow M$ be a principal $S^1$-bundle. Fix a connection form $A \in \Omega^1(Z)$, and let $X_A$ denote the dual vector field, which infinitesimally generates the $S^1$-action. Let $L_A \in \Omega^{\text{ch}}(Z)$ be the vertex operator corresponding to $X_A$. Identifying the Lie algebra of $S^1$ with $i\mathbb{R}$, the non-negative modes of $L_A$ represent the Lie algebra $i\mathbb{R}[t]$ on $\Omega^{\text{ch}}(U)$ for any open set $U\subset Z$. Then $\Omega^{\text{ch}}(U)^{i\mathbb{R}[t]}$ is just the commutant $\text{Com}(L_A,\Omega^{\text{ch}}(U))$. Since $L_A$ commutes with itself, it lies in the center of $\Omega^{\text{ch}}(U)^{i\mathbb{R}[t]}$, and we may take the quotient $\Omega^{\text{ch}}(U)^{i\mathbb{R}[t]} / \langle L_A\rangle$ by the ideal generated by $L_A$. The assignment $$U \mapsto  \Omega^{\text{ch}}(U)^{i\mathbb{R}[t]} / \langle L_A \rangle$$ defines a sheaf of vertex algebras on $Z$ which we denote by $(\Omega^{\text{ch}}_Z)^{i\mathbb{R}[t]} / \langle L_A \rangle$.


There is an induced differential $D$ on $(\Omega^{\text{ch}}_Z)^{i\mathbb{R}[t]} / \langle L_A \rangle$, and a surprising result is that $H^*((\Omega^{\text{ch}}(Z)^{i\mathbb{R}[t]} / \langle L_A \rangle),D)$ does {\it not} vanish in positive weight. If $Z$ has finite topological type, meaning that it has a finite covering by contractible open sets, we will show that \begin{equation} \label{intro:cohomologyquotient} H^*((\Omega^{\text{ch}}(Z)^{i\mathbb{R}[t]} / \langle L_A \rangle),D) \cong H^*(Z) \otimes \cF,\end{equation} where $\cF$ denotes the {\it symplectic fermion algebra}. Note that the higher-weight components depend only on a fixed vertex algebra $\cF$ that is independent of $Z$, and thus capture no additional topological information. If $Z$ has finite topological type, the graded character of $H^*((\Omega^{\text{ch}}(Z)^{i\mathbb{R}[t]} / \langle L_A \rangle),D)$ can be written down immediately from \eqref{intro:cohomologyquotient}. We remark that the symplectic fermion theory was one of the first logarithmic conformal field theories to appear in the physics literature \cite{Kausch}. The orbifold $\cF^{\mathbb{Z}/2\mathbb{Z}}$ coincides with the simplest triplet vertex algebra $\cW_{2,1}$ which was one of the first examples of a non-rational, $C_2$-cofinite vertex algebra \cite{Abe,AM}.


We may also regard $(\Omega^{\text{ch}}_Z)^{i\mathbb{R}[t]} / \langle L_A \rangle$ as a $\mathbb{Z}/2\mathbb{Z}$-graded sheaf of modules over itself, and we use the notation $(\Omega^{\text{ch},\bullet}_Z)^{i\mathbb{R}[t]} / \langle L_A \rangle$. It has an induced twisted differential $D_H$, and if $Z$ has finite topological type we will show that $$H^{\bullet}((\Omega^{\text{ch}}(Z)^{i\mathbb{R}[t]} / \langle L_A \rangle),D_H) \cong H^{\bullet}(Z,H) \otimes \cF.$$ Again, the positive-weight components carry no new topological information, and the graded character can be written down.

We also consider the $H$-twisted version of this sheaf, namely $(\Omega^{\text{ch},H}_Z)^{i\mathbb{R}[t]} / \langle \tilde{L}_A - \widetilde{\iota_A H} \rangle$. Here $\tilde{L}_A - \widetilde{\iota_A H}$ corresponds to $L_A$ under the untwisting isomorphism \eqref{intro-untwisting}. We have an isomorphism of sheaves of differential graded vertex algebras \begin{equation} \label{intro-untwistinginv} (((\Omega^{\text{ch}}_Z)^{i\mathbb{R}[t]} / \langle L_A \rangle),D) \ra (((\Omega^{\text{ch},H}_Z)^{i\mathbb{R}[t]} / \langle \tilde{L}_A - \widetilde{\iota_A H} \rangle), \tilde{D}), \end{equation} and an isomorphism of sheaves of modules $$(((\Omega^{\text{ch},\bullet}_Z)^{i\mathbb{R}[t]} / \langle L_A \rangle), D_H) \ra (((\Omega^{\text{ch},H,\bullet}_Z)^{i\mathbb{R}[t]} / \langle \tilde{L}_A - \widetilde{\iota_A H} \rangle), \tilde{D}_H) .$$ 

The chiral analogue of the Cavalcanti-Gualtieri isomorphism of Courant algebroids \eqref{tduality-CG} is an isomorphism of vertex algebra sheaves on $M$
\begin{equation}\label{tduality-chiral} \tau^{\text{ch}}: \pi_*\left((\Omega^{\text{ch}, H}_Z)^{i\mathbb{R}[t]} / \langle \tilde{L}_A - \widetilde{\iota_A H} \rangle \right) \ra \widehat{\pi}_* \left((\Omega^{\text{ch},\widehat{H}}_{\widehat Z})^{i\mathbb{R}[t]} / \langle \tilde{L}_{\widehat{A}} - \widetilde{\iota_{\widehat{A}} \widehat{H}} \rangle\right).\end{equation} Note that this map has no degree shift. By composing it on the left and right by the isomorphisms \eqref{intro-untwistinginv}, we obtain an isomorphism of {\it untwisted} vertex algebra sheaves
\begin{equation}\label{tduality-chiraluntwist} \tau^{\text{ch}}: \pi_*\left((\Omega^{\text{ch}}_Z)^{i\mathbb{R}[t]} / \langle L_A \rangle \right) \ra \widehat{\pi}_* \left((\Omega^{\text{ch}}_{\widehat Z})^{i\mathbb{R}[t]} / \langle L_{\widehat{A}} \rangle \right).\end{equation} which we also denote by $\tau^{\text{ch}}$. There is a compatible isomorphism of sheaves of modules
 \begin{equation} \label{tduality-chiralmodule} T^{\text{ch}}: \pi_*\left((\Omega^{\text{ch}, \bullet}_Z)^{i\mathbb{R}[t]}/ \langle L_A \rangle\right) \ra  \widehat{\pi}_*\left((\Omega^{\text{ch}, \bullet+1}_{\widehat Z})^{i\mathbb{R}[t]} / \langle L_{\widehat{A}} \rangle \right), \end{equation} which incorporates the degree shift relevant to T-duality. The isomorphism of global sections  \begin{equation} \label{tduality-chiralmoduleuntwistglobal} T^{\text{ch}}: \Omega^{\text{ch}, \bullet}(Z) ^{i\mathbb{R}[t]}/  \langle L_A \rangle \ra  \Omega^{\text{ch}, \bullet+1}({\widehat Z})^{i\mathbb{R}[t]} / \langle L_{\widehat{A}} \rangle, \end{equation} is the analogue of \eqref{tduality-chain}. At weight zero, it coincides with \eqref{tduality-chain} and induces the cohomology isomorphism \eqref{tduality-coh}. Even though \eqref{tduality-chiralmoduleuntwistglobal} does not intertwine the differentials $D_H$ and $D_{\widehat{H}}$ in positive weight, when $Z$ has finite topological type we still get a linear isomorphism \begin{equation} \label{tduality-chiralcohomology} H^{\bullet} ((\Omega^{\text{ch}}(Z) ^{i\mathbb{R}[t]}/  \langle L_A \rangle), D_H) \cong H^{\bullet+1}((\Omega^{\text{ch}}({\widehat Z})^{i\mathbb{R}[t]} / \langle L_{\widehat{A}} \rangle), \widehat{D}_H).\end{equation}
 
The isomorphisms \eqref{tduality-chiraluntwist}, \eqref{tduality-chiralmoduleuntwistglobal}, and \eqref{tduality-chiralcohomology} are interpreted as a more refined version of T-duality from a loop space perspective. For an alternate loop space perspective of T-duality, see \cite{HM14}. For a relation of current algebras to T-duality, see \cite{Hekmati}. An example of T-duality using OPE in the case of the trivial circle bundle on a 2D torus with nontrivial H-flux is studied in \cite{AH}.



\subsection{Outline of paper}

In \S\ref{sect:courant} we briefly review Courant algebroids and describe the main examples we need, which are the exact Courant algebroids, and the $S^1$-invariant Courant algebroids associated to a principal $S^1$-bundle. In \S\ref{sect:vertex} we introduce vertex algebras and describe the examples and constructions we need, including orbifolds, commutants, and modules. In \S\ref{sect:cdr} we recall the chiral de Rham complex and give a coordinate-free construction by strong generators and relations. We also construct the twisted chiral de Rham differential $D_H$ and prove the positive weight vanishing theorem for $H^{\bullet}(\Omega^{\text{ch}}(Z), D_H)$. In  \S\ref{sect:courantsheaves}, we recall Heluani's theorem that associates to any Courant algebroid $E\rightarrow Z$ a sheaf of vertex algebras $\cU^E$ on $Z$. In the case $E = (TZ \oplus T^*Z, [\cdot,\cdot]_H)$, we denote $\cU^E$ by $\Omega^{\text{ch},H}_Z$ and we find generators and relations for $\Omega^{\text{ch},H}_Z$, and establish the untwisting isomorphism \eqref{intro-untwisting} for any $H$. We also find generators and relations for $(\Omega^{\text{ch}}_Z)^{i \mathbb{R}[t]} / \langle L_A \rangle$ as well as its twisted counterpart, and we establish the relevant untwisting theorems. In \S\ref{sect:T-duality} we review T-duality following the setup and notation from \cite{BEM}. Finally, in \S\ref{sect:chiralT-duality}, we construct the vertex algebra analogues of  \eqref{tduality-coh}, \eqref{tduality-chain} and \eqref{tduality-CG}.

\tableofcontents

\section{Generalized geometry, Courant algebroids and Courant morphisms}\label{sect:courant}

\subsection{Courant algebroids}
A {\it Courant algebroid} is a vector bundle $E \ra Z$ equipped with a nondegenerate, symmetric bilinear form $$\langle \cdot ,\cdot \rangle: C^{\infty}(E) \otimes  C^{\infty}(E) \ra C^{\infty}(Z),$$  a skew-symmetric bracket $$[\cdot ,\cdot]: C^{\infty}(E) \otimes  C^{\infty}(E) \ra C^{\infty}(E),$$ and a smooth bundle map $\pi: E \ra TZ$ called the {\it anchor}. There is an induced differential operator $d: C^{\infty}(Z) \ra C^{\infty}(E)$ given by $$\langle df, A\rangle = \frac{1}{2} \pi(A) f.$$ For all $A,B,C \in C^{\infty}(E)$ and $f\in C^{\infty}(Z)$, these structures satisfy
\begin{enumerate} \item $\pi([A,B]) = [\pi(A),\pi(B)]$.
\item $\text{Jac}(A,B,C) = d(\text{Nij}(A,B,C))$. Here $$\text{Jac}(A,B,C) = [[A,B],C] + [[B,C],A] + [[C,A],B],$$ $$\text{Nij}(A,B,C) = \frac{1}{3}(\langle [A,B],C\rangle + \langle [B,C],A\rangle + \langle [C,A],B\rangle.$$ 
\item $\pi \circ d = 0$.
\item $[A,fB] = (\pi(A)f)B + f[A,B] - \langle A,B\rangle df$.
\item $\pi(A) \langle B,C\rangle = \langle [A,B] + d(\langle A,B\rangle),C \rangle + \langle B,[A,C] + d(\langle A,C\rangle) \rangle$.
\end{enumerate}
Notice that there is a dual to the anchor map $\pi^*: T^*Z \to E^*\cong E$, where the isomorphism is given by $\langle \cdot ,\cdot \rangle$. It follows that the bracket in general does not satisfy the Jacobi identity, so is not a Lie bracket.

\begin{remark}
The notion of a Courant algebroid was introduced by Liu-Weinstein-Xu \cite{LWX} so as to provide a framework for the theory of 
Courant and Dorfman \cite{Courant,Dorfman}. The original definition was later simplified by Roytenberg \cite{Roytenberg}  and others.
\end{remark}

A Courant algebroid over a point is a quadratic Lie algebra, that is a Lie algebra with an invariant nondegenerate bilinear form.
The so called {\em standard} Courant algebroid is $E= TZ \oplus T^*Z$, where $\pi$ is the projection, and $$ \langle X + \xi, Y + \eta\rangle = \frac{1}{2}( \iota_X \eta + \iota_Y \xi),$$ $$[X + \xi, Y + \eta] = [X,Y] + L_X\eta - L_Y \xi - \frac{1}{2} d(\iota_X \eta - \iota_Y \xi).$$
$E$ is also known as the {\em generalized tangent bundle} in generalized geometry.

More generally, one considers {\em exact} Courant algebroids $E$ in generalized geometry. That is, there is an exact sequence of vector bundles
$$0\ra T^*Z \ra E \ra TZ\ra 0,$$ where the last map is $\pi$. Exact Courant algebroids are classified by $H^3(Z,\mathbb{R})$, \cite{Severa1,Severa2}. For every such $E$, there is a splitting $E \cong TZ \oplus T^*Z$ and a closed $3$-form $H \in \Omega^3(Z)$ such that the bilinear form and bracket are given by
\begin{equation} \label{twistedbracket} [X + \xi, Y + \eta]_H = [X,Y] + L_X\eta - L_Y \xi - \frac{1}{2} d(\iota_X \eta - \iota_Y \xi)+ \iota_X \iota_Y H.\end{equation}
For any $H$, $\Omega^*(Z)$ is a Clifford module over $TZ\oplus T^*Z$ via \begin{equation} \label{cliffordmod} (X+\xi)\cdot \omega = \iota_X (\omega) + \xi \wedge \omega.\end{equation} Define the {\it twisted de Rham differential} $d_H$ on $\Omega^*(Z)$ by $$d_H(\omega) = d(\omega) + H\wedge \omega.$$ Note that $d_H$ is not a derivation of $\Omega^*(Z)$ regarded as an algebra, but it is a derivation of $\Omega^*(Z)$ regarded as a left $\Omega^*(Z)$-module. In particular, for homogeneous $a,\omega \in \Omega^*(Z)$, we have $d_H(a\omega) = d(a) \omega + (-1)^{|a| |\omega|} a d_H(\omega)$. For $X,Y \in \text{Vect}(Z)$, $\xi,\eta \in \Omega^1(Z)$, and $\omega \in \Omega^*(Z)$, we have
\begin{equation} \label{twistedcompat} [X+\xi, Y+\eta]_H \cdot \omega = [[d_H,X+\xi],Y+\eta]\cdot \omega. \end{equation} Finally, note that $d_H$ is not homogeneous with respect to degree. We may regard $\Omega^*(Z)$ as a $\mathbb{Z}/2\mathbb{Z}$-graded complex by reducing the degree modulo 2, and we use the notation $\Omega^{\bullet}(Z)$. Then $d_H$ maps $\Omega^{\bullet}(Z)$ to $\Omega^{\bullet+1}(Z)$.

There is one more class of Courant algebroids that we need to consider. Suppose that $\pi: Z \ra M$ is a principal circle bundle. Let $E$ denote the quotient $(TZ \oplus T^*Z)/{S^1}$ of the standard Courant algebroid on $Z$ by the $S^1$-action. It is well known that $E$ is a Courant algebroid which is not exact, but is transitive; that is, the anchor map is surjective. More generally, given a closed $3$-form $H$ on $Z$, $E = (TZ \oplus T^*Z, [\cdot , \cdot]_H)/{S^1}$ is transitive. Via the {\it dimension reduction} formalism, we shall regard $E$ as a Courant algebroid on $M$ rather than $Z$. Fix a connection form $A \in \Omega^1(Z)$, and let $X_A$ denote the vector field on $Z$ generated by the $S^1$-action, normalized so that \begin{equation} \label{normalization} \iota_{A}A = 1.\end{equation} Here $\iota_A$ denotes the contraction along $X_A$. Let $\text{Vect}_{\text{hor}}(Z)$ denote the set of $S^1$-invariant vector fields on $Z$ such that $\iota_X A = 0$, which under $\pi_*$ can be identified with $\text{Vect}(M)$. Then a vector field in $TZ/S^1$ can be written as \begin{equation} \label{dimreduction1} (X, f) = X + f X_A,\qquad X \in \text{Vect}(M), \qquad f \in C^{\infty}(M).\end{equation} Similarly, an element of $(T^*Z)/S^1$ can be written as \begin{equation} \label{dimreduction2} (\omega,g) = \omega + g A,\qquad \omega \in \Omega^1(M), \qquad g \in C^{\infty}(M).\end{equation} The anchor map $E \ra TM$ is just the composition of the anchor map $E \ra TZ$ with $\pi_*$. Writing $H = H^{(3)} + A \wedge H^{(2)}$ and $F_A = dA$, the following expression for the bracket $[\cdot, \cdot]_H$ is given in \cite{Bouwknegt07}.
\begin{multline} \label{dimredbracket}  [(X_1, f_1) + (\omega_1, g_1), (X_2, f_2) + (\omega_2, g_2)]_H = \bigg([X_1,X_2], X_1(f_2) - X_2(f_1) + \iota_{X_1} \iota_{X_2} F_A\bigg) \\
+ \bigg(\text{Lie}_{X_1}(\omega_2) - \text{Lie}_{X_2}(\omega_1) + (g_2 \iota_{X_1} F_A - g_1 \iota_{X_2} F_A) - \frac{1}{2} d(\iota_{X_1} \omega_2 - \iota_{X_2} \omega_1)\\
+\frac{1}{2} (df_1 g_2 + f_2 dg_1 - f_1 dg_2 - df_2 g_1) + \iota_{X_1} \iota_{X_2} H^{(3)} + (f_2 \iota_{X_1} H^{(2)} - f_1 \iota_{X_1} H^{(2)}), X_1(g_2) - X_2(g_1) + \iota_{X_1} \iota_{X_2} H^{(2)}\bigg). \end{multline} 



\subsection{Courant morphisms}

Suppose $E\to Z$ and $E'\to Z'$ are Courant algebroids. Denote by $E^{op}$ the Courant algebroid $E$ with 
the opposite inner product. Given a smooth map 
$\Phi\colon Z\to Z'$ let 
\[ 
\text{Graph}(\Phi)=\{(\Phi(m),m)|m\in Z\}\subset Z'\times Z\]
be its graph relation. A {\em Courant
    morphism} $R_\Phi\colon E\to E'$ is a smooth map $\Phi\colon
  Z\to Z'$ together with a subbundle
\[ R_\Phi\subset (E'\times {E}^{op})|_{\text{Graph}(\Phi)}\]
with the following properties: 
\begin{enumerate}
\item $R_\Phi$ is a Lagrangian submanifold (i.e. $R_\Phi^\perp=R_\Phi$). 
\item For $e'\in E', \, e\in E$, the image $(e'\times e )(R_\Phi)$ is tangent to 
      $\text{Graph}(\Phi)$. 
\item Any pair of  sections of $E'\times E^{op}$ that restrict to
      sections of $R_\Phi$, then so does their Courant bracket.
\end{enumerate}
The composition of Courant morphisms is given as the fibrewise
composition of relations, where  the usual transversality conditions
are imposed to ensure smoothness.  

\begin{remark}
The notion of a Courant morphism is due to {\v{S}}evera
\cite{Severa3} and Alekseev-Xu
\cite{AX}, see also \cite{Li-Bland} for a published account.
\end{remark}

\begin{enumerate}
\item 
A  diffeomorphism $\Phi$ of $Z$ clearly gives rise to an automorphism  $R_\Phi$ of the 
standard Courant algebroid $TZ \oplus T^*Z$. Let $B \in \Omega^2(Z)$ be a 2-form on $Z$.
Then $e^B(X+\xi) = X+\xi + i_X B$ is an automorphism of the 
standard Courant algebroid if and only if $B$ is closed. More generally, the semidirect product 
$\text{Diff}(Z)\ltimes \Omega^2_{\text{cl}}(Z)$ is the automorphism group of the 
standard Courant algebroid, where $\Omega^2_{\text{cl}}(Z)$ denotes the space of closed $2$-forms \cite{Hitchin2}.\\

\item Given an exact Courant algebroid $(TZ \oplus T^*Z, [\cdot , \cdot]_H)$ where 
$H$ is a closed 3-form on $Z$,  then any diffeomorphism $\Phi$ of $Z$ 
satisfying $\Phi^*(H)=H$ clearly gives rise to an automorphism  $R_\Phi$.
More generally, the semidirect product 
$\text{Diff}_H(Z)\ltimes \Omega^2_{\text{cl}}(Z)$ is contained in the automorphism group of the 
exact Courant algebroid.\\

\item We give an example of a Courant morphism that is not an automorphism. Let $0<n<N$ be integers
 and let $S$ be an $n$-dimensional totally geodesic submanifold of the $N$-dimensional Riemannian manifold $Z$. 
 It is a straightforward fact that vector fields on a Riemannian manifold
 induce vector fields on totally geodesic submanifolds compatible with the bracket, using the Levi-Civita
 connections on the manifold and its totally geodesic submanifold. Together with the restriction map 
 on differential forms, we get a morphism $R_\iota\colon TZ\oplus T^*Z \longrightarrow TS\oplus T^*S$ of 
 standard Courant algebroids, where $\iota\colon S\longhookrightarrow Z$ is the embedding. 
\end{enumerate}

\section{Vertex algebras}\label{sect:vertex}

In this section, we define vertex algebras, which have been discussed from various points of view in the literature (see for example \cite{Borcherds,FLM,K,FBZ}). We will follow the formalism developed in \cite{LZ2} and partly in \cite{LiI}. Let $V=V_0\oplus V_1$ be a super vector space over $\mathbb{C}$, and let $z,w$ be formal variables. By $\text{QO}(V)$, we mean the space of all linear maps $$V\ra V((z))=\{\sum_{n\in\mathbb{Z}} v(n) z^{-n-1}|
v(n)\in V,\ v(n)=0\ \text{for} \ n>\!\!>0 \}.$$ Each $a\in \text{QO}(V)$ can be represented as a power series
$$a=a(z)=\sum_{n\in\mathbb{Z}}a(n)z^{-n-1}\in \text{End}(V)[[z,z^{-1}]].$$ Each $a\in
\text{QO}(V)$ is assumed to be of the form $a=a^0+a^1$ where $a^i:V_j\ra V_{i+j}((z))$ for $i,j\in\mathbb{Z}/2\mathbb{Z}$, and we write $|a^i| = i$.

There are nonassociative bilinear operations
$\circ_n$ on $\text{QO}(V)$, indexed by $n\in\mathbb{Z}$, which we call the $n^{\text{th}}$ circle
products. For homogeneous $a,b\in \text{QO}(V)$, they are defined by
$$ a(w)\circ_n b(w)=\text{Res}_z a(z)b(w)\ \iota_{|z|>|w|}(z-w)^n-
(-1)^{|a||b|}\text{Res}_z b(w)a(z)\ \iota_{|w|>|z|}(z-w)^n.$$
Here $\iota_{|z|>|w|}f(z,w)\in\mathbb{C}[[z,z^{-1},w,w^{-1}]]$ denotes the
power series expansion of a rational function $f$ in the region
$|z|>|w|$. We usually omit $\iota_{|z|>|w|}$ and just
write $(z-w)^{-1}$ to mean the expansion in the region $|z|>|w|$,
and write $-(w-z)^{-1}$ to mean the expansion in $|w|>|z|$. For $a,b\in \text{QO}(V)$, we have the following identity of formal power series, known as the {\it operator product expansion} (OPE) formula.
 \begin{equation}\label{opeform} a(z)b(w)=\sum_{n\geq 0}a(w)\circ_n
b(w)\ (z-w)^{-n-1}+:a(z)b(w):. \end{equation}
Here $:a(z)b(w):\ =a(z)_-b(w)\ +\ (-1)^{|a||b|} b(w)a(z)_+$, where $a(z)_-=\sum_{n<0}a(n)z^{-n-1}$ and $a(z)_+=\sum_{n\geq
0}a(n)z^{-n-1}$. We often write
$$a(z)b(w)\sim\sum_{n\geq 0}a(w)\circ_n b(w)\ (z-w)^{-n-1},$$ where
$\sim$ means equal modulo the term $:a(z)b(w):$, which is regular at $z=w$. 

Note that $:a(w)b(w):$ is a well-defined element of $\text{QO}(V)$. It is called the {\it Wick product} of $a$ and $b$, and it
coincides with $a\circ_{-1}b$. The other negative circle products are given by
$$ n!\ a(z)\circ_{-n-1}b(z)=\ :(\partial^n a(z))b(z):,\qquad \partial = \frac{d}{dz}.$$
For $a_1(z),\dots ,a_k(z)\in \text{QO}(V)$, the iterated Wick product is defined to be
\begin{equation}\label{iteratedwick} :a_1(z)a_2(z)\cdots a_k(z):\ =\ :a_1(z)b(z):,\qquad b(z)=\ :a_2(z)\cdots a_k(z):.\end{equation}
We often omit the variables $z,w$ when no confusion can arise.

The set $\text{QO}(V)$ is a nonassociative algebra with the operations $\circ_n$, which satisfy $1\circ_n a=\delta_{n,-1}a$ for
all $n$, and $a\circ_n 1=\delta_{n,-1}a$ for $n\geq -1$. A subspace $\cA\subset \text{QO}(V)$ containing $1$ which is closed under the circle products will be called a {\it quantum operator algebra} (QOA). A subset $S=\{a_i|\ i\in I\}$ of $\cA$ is said to {\it generate} $\cA$ if every $a\in\cA$ can be written as a linear combination of nonassociative words in the letters $a_i$, $\circ_n$, for $i\in I$ and $n\in\mathbb{Z}$. We say that $S$ {\it strongly generates} $\cA$ if every $a\in\cA$ can be written as a linear combination of words in the letters $a_i$, $\circ_n$ for $n<0$. Equivalently, $\cA$ is spanned by \begin{equation} \label{wickrelations} \{ :\partial^{k_1} a_{i_1}\cdots \partial^{k_m} a_{i_m}:| \ i_1,\dots,i_m \in I,\ k_1,\dots,k_m \geq 0\}.\end{equation} We say that $a,b\in \text{QO}(V)$ are {\it local} if if $(z-w)^N [a(z),b(w)]=0$ for some $N\geq 0$. Here $[\cdot ,\cdot]$ denotes the super bracket. A {\it vertex algebra} may be defined as a QOA whose elements are pairwise local. This definition is well known to be equivalent to the notion of a vertex algebra in \cite{FLM}. 

A very useful description of a vertex algebra $\cA$ is a strong generating set $\{a_i|\ i\in I\}$ for $\cA$, together with a set of generators $\{b_k|\ k\in K\}$ for the ideal $\cI$ of relations among the generators and their derivatives, that is, all expressions of the form \eqref{wickrelations} that vanish. Given such a description, to define a homomorphism $\phi$ from $\cA$ to another vertex algebra $\cB$, it suffices to define $\phi(a_i)$ for $i\in I$ and show the following.
\begin{enumerate}
\item $\phi$ preserves pairwise OPEs among the generators; equivalently, $\phi(a_i \circ_n a_j) = \phi(a_i) \circ_n \phi(a_j)$ for all $i,j \in I$ and $n\geq 0$.
\item $\phi(b_k) = 0$ for all $k \in K$.
\end{enumerate}
This shall be our method of constructing vertex algebra homomorphisms in this paper.

A {\it conformal structure} is an element $L(z) = \sum_{n\in \mathbb{Z}} L_n z^{-n-2}$ in a vertex algebra $\cA$ satisfying $$L(z) L(w) \sim \frac{c}{2}(z-w)^{-4} + 2L(w) (z-w)^{-1} + \partial L(w)(z-w)^{-1},$$ such that $L_{-1}$ acts by $\partial$ on $\cA$ and $L_0$ acts diagonalizably. The constant $c$ is called the {\it central charge}, and the grading by $L_0$-eigenvalue is called {\it conformal weight}. In all our examples, the conformal weight grading is by the nonnegative integers. In the presence of a conformal structure, we write a homogeneous element $a(z) =  \sum_{n\in Z} a(n) z^{-n-1}$ in the form $\sum_{n\in Z} a_n z^{-n-\text{wt}(a)}$ where $a_n = a(n+\text{wt}(a)-1)$.

A {\it module} $\cM$ over a vertex algebra $\cA$ is a vector space $\cM$ together with a QOA homomorphism $\cA \ra \text{QO}(\cM)$. In particular, for each $a\in \cA$, we have a field $a_{\cM}(z) = \sum_{n\in \mathbb{Z}} a_{\cM}(n) z^{-n-1}$ where $a_{\cM}(z) \in \text{End}(\cM)$. If $\cA$ and $\cM$ are graded by conformal weight, we write $a_{\cM}(z) = \sum_{n\in \mathbb{Z}} a_{\cM,n} z^{-n-\text{wt}(a)}$, and we require that $a_{\cM,n}$ has weight $- n$. 

The {\it orbifold} construction is a standard way to construct new vertex algebras from old one. Given a vertex algebra $\cA$ and a group $G$ of automorphisms of $\cA$, the invariant subalgebra $\cA^G$ is called an orbifold. Many interesting vertex algebras arise either as orbifolds or as extensions of orbifolds; the spectacular {\it moonshine vertex algebra} is an important example \cite{FLM}.

The {\it commutant} or {\it coset construction} is another way to construct new vertex algebras from old ones. Let $\cV$ be a vertex algebra and let $S$ be a subalgebra of $\cV$. The commutant $\text{Com}(S, \cV)$ is defined to be $\{v\in \cV|\ [a(z),v(w)] = 0,\ \forall a\in S\}$, which is always a vertex subalgebra of $\cV$. Equivalently, $a\circ_k v = 0$ for all $k\geq 0$ and $a\in S$. If $\cA$ is the vertex algebra generated by $S$, $\text{Com}(S, \cV) = \text{Com}(\cA, \cV)$.

\begin{example} ($\beta\gamma$ and $bc$ systems) Let $V$ be a finite-dimensional vector space. Regard
$V\oplus V^*$ as an abelian Lie algebra. Then its loop algebra has a one-dimensional central extension
by $\mathbb{C} \kappa $ which we denote by $\gh$, with bracket
$$[(x,x')t^n,(y,y')t^m]=(\bra y',x\ket-\bra x',y\ket)\delta_{n+m,0}\kappa.$$
Let $\gb\subset \gh$ be the subalgebra generated by $\kappa$, $(x,0)t^n$, $(0,x')t^{n+1}$, for $n\geq 0$, and let $\C$ be the one-dimensional $\gb$-module on which each $(x,0)t^n$, $(0,x')t^{n+1}$ act trivially and the central element $\kappa$ acts by
the identity. Consider the $U(\gh)$-module $U(\gh)\otimes_{U(\gb)}\C$. The operators representing $(x,0)t^n,(0,x')t^{n}$ on this module are denoted by $\beta^x_n,\gamma^{x'}_n$, and the fields
$$\beta^x(z)=\sum\beta^x_n z^{-n-1},\qquad \gamma^{x'}(z)=\sum\gamma^{x'}_n z^{-n}\in
\text{QO} (U(\gh)\otimes_{U(\gb)}\C)$$
have the properties
$$ [\beta^x_+(z),\gamma^{x'}(w)]=\bra x',x\ket(z-w)^{-1},~~~[\beta^x_-(z),\gamma^{x'}(w)]=\bra x',x\ket(w-z)^{-1}.$$
It follows that $(z-w)[\beta^x(z),\gamma^{x'}(w)]=0$. Moreover the $\beta^x(z)$ commute; likewise for the $\gamma^{x'}(z)$. Thus the $\beta^x(z),\gamma^{x'}(z)$ generate a vertex algebra $\cS = \cS(V)$.
This algebra was introduced in \cite{FMS}, and is known as a $\beta\gamma$-system or semi-infinite symmetric algebra. Finally, $\beta^x,\gamma^{x'}$ satisfy the OPE relation
$$ \beta^x(z)\gamma^{x'}(w)\sim\bra x',x\ket(z-w)^{-1}. $$
By the Poincar\'e-Birkhoff-Witt Theorem, the vector space $U(\gh)\otimes_{U(\gb)}\C$ has the structure of a polynomial algebra
with generators $\beta^x_n,\gamma^{x'}_{n+1}$, $n<0$, which are linear in $x\in V$ and $x'\in V^*$.

We can also regard $V\oplus V^*$ as an odd abelian Lie (super) algebra, and consider its loop algebra and a one-dimensional central extension by $\C\tau$ with bracket
$$[(x,x')t^n,(y,y')t^m]=(\bra y',x\ket+\bra x',y\ket)\delta_{n+m,0}\tau.$$
Call this Lie algebra $\gj$, and form the induced module $U(\gj)\otimes_{U(\ga)}\C$. Here $\ga$ is the subalgebra of $\gj$ generated by $\tau$, $(x,0)t^n$, $(0,x')t^{n+1}$, for $n\geq 0$, and $\C$ is the one-dimensional $\ga$-module on
which $(x,0)t^n$, $(0,x')t^{n+1}$ act trivially and $\tau$ acts by $1$. Then there is a vertex algebra $\cE = \cE(V)$, analogous to
$\cS$, and generated by odd vertex operators $b^x(z) = \sum_{n\in \mathbb{Z}} b^x_n z^{-n-1}$ and $c^{x'}(z) = \sum_{n\in \mathbb{Z}} c^{x'}_n z^{-n}$ in $\text{QO}(U(\gj)\otimes_{U(\ga)}\C)$, with OPE relations
$$b^x(z)c^{x'}(w)\sim\bra x',x\ket(z-w)^{-1}.$$
This vertex algebra is known as a $bc$-system, or a semi-infinite exterior algebra. The vector space $U(\gj)\otimes_{U(\ga)}\C$ has the structure of an odd polynomial algebra with generators $b^x_n ,c^{x'}_{n+1}$, $n<0$, which are linear in $x\in V$ and $x'\in V^*$.

The $bc\beta\gamma$-system of rank $n$ is just $\cE\otimes \cS$. Often, we fix a basis for $V$ and a dual basis for $V^*$, so that the generators of $\cE\otimes \cS$ are $b^i, c^i, \beta^i, \gamma^i$ satisfying the nontrivial OPEs $\beta^i(z) \gamma^i(w) \sim \delta_{i,j}(z-w)^{-1}$ and $b^i(z) c^j(w) \sim \delta_{i,j}(z-w)^{-1}$.
\end{example}

\section{The chiral de Rham complex}\label{sect:cdr}
The chiral de Rham complex $\Omega^{\text{ch}}_Z$ is a sheaf of vertex algebras on any smooth manifold $Z$ in either the algebraic, complex, or smooth categories, which was introduced by Malikov, Schechtman, and Vaintrob \cite{MSV1}. In this paper we work exclusively in the smooth category, and we will use the formulation in \cite{LL} which is equivalent but more convenient for our purposes. In fact, the smooth chiral de Rham complex is not quite a sheaf, but it is a {\it weak} sheaf in the terminology of \cite{LLS}. This means that the reconstruction axiom holds only for finite open covers. However, it is graded by conformal weight and each weighted subspace is an ordinary sheaf. Since we always work inside a fixed weighted component, this does not cause problems and we shall drop the word \lq\lq weak" throughout this paper. 

For a coordinate open set $U\subset \mathbb{R}^n$ with coordinate functions $\gamma^i$, $i=1,\dots, n$, the algebra of sections $\Omega^{\text{ch}}(U)$ has odd generators $b^i(z) = \sum_{n\in \mathbb{Z}} b^i_n z^{-n-1}$ and $c^i(z) = \sum_{n\in \mathbb{Z}} c^i_n z^{-n}$, even generators $\beta^i(z) = \sum_{n\in \mathbb{Z}} \beta^i_n z^{-n-1}$, as well as an even generator $f(z) = \sum_{n\in \mathbb{Z}} f_n z^{-n}$ for every smooth function $f = f(\gamma^1,\dots, \gamma^n) \in C^{\infty}(U)$. The element $\beta^i$ corresponds to the vector field $\frac{\partial}{\partial \gamma^i}$, $c^i$ corresponds to the one-form $d\gamma^i$, and $b^i$ corresponds to the contraction operator along $\frac{\partial}{\partial \gamma^i}$. The fields $f$ commute with $b^i$ and $c^i$, and satisfy $$\beta^i(z) f(w) \sim \frac{\partial f}{\partial \gamma^i}(w)(z-w)^{-1},$$ which generalizes the formula $\beta^i(z)\gamma^j(w) \sim \delta_{i,j}(z-w)^{-1}$. These OPE relations define a Lie conformal algebra \cite{K}, and $\Omega^{\text{ch}}(U)$ is defined as the quotient of the corresponding universal enveloping vertex algebra by the ideal generated by
\begin{equation} \label{cdrrelations} \partial f - \sum_{i=1}^n :\frac{\partial f}{\partial x^i} \partial \gamma^i:,\qquad :fg: \ - fg, \qquad 1- \text{id}.\end{equation}
A typical element of $\Omega^{\text{ch}}(U)$ is a linear combination of fields of form
$$ :f \partial^{a_1} b^{i_1} \cdots \partial^{a_r} b^{i_r} \partial^{d_1} c^{j_1} \cdots \partial^{d_s} c^{j_s} \partial^{e_1} \beta^{k_1} \cdots \partial^{e_t} \beta^{k_t} \partial^{m_1} \gamma^{l_1} \cdots \partial^{m_u} \gamma^{l_u}:,$$ where $a_i, d_i, e_i \geq 0$ and $m_i\geq 1$.

Now consider a smooth change of coordinates $g:U \ra U'$, $$\tilde{\gamma}^i = g^i(\gamma) = g^i(\gamma^1,\dots, \gamma^n),\qquad \gamma^i = f^i(\tilde{\gamma})= f^i(\tilde{\gamma}^1,\dots, \tilde{\gamma}^n).$$ We get the following transformation rules:
\begin{equation} \begin{split} \label{cdrlocaltransform} \tilde{c}^i = \ :\frac{\partial g^i}{\partial \gamma^j} c^j:, \qquad  \tilde{b}^i = \ :\frac{\partial f^j}{\partial \tilde{\gamma}^i} (g(\gamma)) b^j:,\\ \tilde{\beta}^i =\ :\beta^j \frac{\partial f^j}{\partial \tilde{\gamma}^i}(g(\gamma)): + :\frac{\partial^2 f^k}{\partial \tilde{\gamma}^i \partial \tilde{\gamma}^l} (g(\gamma)) \frac{\partial g^l}{\partial \gamma^r} c^r b^k:.\end{split} \end{equation}
These new fields satisfy OPE relations
$$\tilde{b}^i(z) \tilde{c}^j(w) \sim \delta_{i,j}(z-w)^{-1},\qquad \tilde{\beta}^i(z) \tilde{f}(w) \sim \frac{\partial \tilde{f}}{\partial \tilde{\gamma}^i} (z-w)^{-1}.$$ Here $\tilde{f} = \tilde{f}(\tilde{\gamma}^1,\dots, \tilde{\gamma}^n)$ is any smooth function. Therefore $g:U \ra U'$ induces a vertex algebra isomorphism $ \phi_g: \Omega^{\text{ch}}(U) \rightarrow \Omega^{\text{ch}}(U')$. Moreover, given diffeomorphisms of open sets $U_1 \xrightarrow{g} U_2 \xrightarrow{h}U_3$, we get
$\phi_{h \circ g}=\phi_{g} \circ \phi_{h}$. This allows us to define the sheaf $\Omega^{\text{ch}}_Z$ on any smooth manifold $Z$, using standard arguments of formal geometry \cite{GK}.

Consider the following locally defined fields
\begin{equation} J = \sum_{i=1}^n :b^i c^i:,\qquad Q = \sum_{i=1}^n :\beta^i c^i:,\qquad G = \sum_{i=1}^n :b^i \partial\gamma^i:, \qquad L = \sum_{i=1}^n :\beta^i \partial \gamma^i: - :b^i \partial c^i:.\end{equation} These satisfy the OPE relations of a topological vertex algebra of rank $n$ \cite{LZ1}.
\begin{equation}\begin{split} 
L(z) \sim L(w) \sim 2L(w)(z-w)^{-2} + \partial L(w)(z-w)^{-1},\\
J(z) J(w) \sim - n (z-w)^{-2},\\
L(z) J(w) \sim -n (z-w)^{-3} + J(w) (z-w)^{-2} +  \partial J(w) (z-w)^{-1}, \\
G(z) G(w) \sim 0,\\ L(z) G(w) \sim  2G(w) (z-w)^{-2} +  \partial G(w) (z-w)^{-1}, \\ J(z) G(w) \sim -G(w) (z-w)^{-1},\\
Q(z) Q(w) \sim 0,\\  L(z) Q(w) \sim  Q(w) (z-w)^{-2} +  \partial Q(w) (z-w)^{-1}, \\ J(z) Q(w) \sim Q(w) (z-w)^{-1},\\
Q(z) G(w) \sim n (z-w)^{-3} + J(w) (z-w)^{-2} + L(w)(z-w)^{-1}.
\end{split}
\end{equation}
Under $g:U \ra U'$, these fields transform as follows:
\begin{equation} \begin{split} 
\tilde{L} =  L,\qquad \tilde{G} = G,\\
\tilde{J} = J + \partial \bigg(\text{Tr}\ \text{log}\bigg(\frac{\partial g^i}{\partial b^j}\bigg)\bigg),\qquad 
\tilde{Q} = Q +  \partial \bigg( \frac{\partial}{\partial \tilde{b}^r} \bigg(\text{Tr}\ \text{log} \bigg( \frac{\partial f^i}{\partial \tilde{b}^j} \bigg) \bigg) \tilde{c}^r \bigg),
\end{split}
\end{equation}
Therefore $L$ and $G$ are globally defined on any manifold $Z$. Although $J$ and $Q$ are not globally defined in general, the operators $J_0$ and $Q_0$ are well-defined. Note that $\Omega^{\text{ch}}(Z)$ has a bigrading by degree and weight, where the weight is the eigenvalue of $L_0$ and degree is the eigenvalue of $J_0$. Also, $Q_0$ is a square-zero operator and we define the differential $D$ to be $Q_0$. It is derivation of all circle products $\circ_k$, $k\in \mathbb{Z}$, and it coincides with the de Rham differential at weight zero. Note that $G_0$ is a contracting homotopy for $D$ in the sense that $[D,G_0] = L_0$. This shows that the cohomology $H^*(\Omega^{\text{ch}}(Z),D)$ vanishes beyond weight zero. Also, note that each $f$ has weight $0$ and degree $0$, $c^i$ has weight $0$ and degree $1$, $b^i$ has weight $1$ and degree $-1$, and $\beta^i$ has weight $1$ and degree $0$. Therefore the weight zero component of $\Omega^{\text{ch}}(Z)$ is just $\Omega(Z)$, and the embedding $\Omega(Z) \hookrightarrow \Omega^{\text{ch}}(Z)$ is a quasi-isomorphism.

\subsection{Commutant and orbifold subsheaves}

Given a global section $\tau\in \Omega^{\text{ch}}(Z)$, for each open set $U\subset Z$, let $\text{Com}(\tau |_U, \Omega^{\text{ch}}(U))$ denote the commutant of $\tau |_U$ inside $\Omega^{\text{ch}}(U)$. By Lemma 3.7 of \cite{LLS}, the commutant condition is local, so $U \mapsto \text{Com}(\tau |_U, \Omega^{\text{ch}}(U))$ defines a sheaf of vertex algebras on $Z$. This means in particular that an element $\omega \in \Omega^{\text{ch}}(Z)$ lies in $\text{Com}(\tau, \Omega^{\text{ch}}(Z))$ if and only if $\omega |_{U} \in \text{Com}(\tau |_U, \Omega^{\text{ch}}(U))$ for all $U\subset Z$. 

We shall also need an analogous result for orbifolds. Suppose that $G$ is a compact, connected Lie group acting on $Z$ by diffeomorphisms. Then $G$ acts on $\Omega^{\text{ch}}(Z)$ by vertex algebra automorphisms and $\Omega^{\text{ch}}(Z)^{G}$ is a vertex subalgebra. As shown in \cite{LL}, the action of $G$ induces a vertex algebra homomorphism $V_0(\mathfrak{g}) \rightarrow \Omega^{\text{ch}}(Z)$, where $V_0(\mathfrak{g})$ denotes the level zero affine vertex algebra of the Lie algebra $\mathfrak{g}$ of $G$. The generators of $V_0(\mathfrak{g})$ are $L_{\xi}(z)$, which are linear in $\xi \in \mathfrak{g}$ and satisfy
$$L_{\xi}(z) L_{\eta}(w) \sim L_{[\xi, \eta]}(w)(z-w)^{-1}.$$ Since $G$ is connected, we may identify $\Omega^{\text{ch}}(Z)^G$ with the invariant space $\Omega^{\text{ch}}(Z)^{\gg}$, which is just the joint kernel of $\{(L_{\xi})_0|\ \xi\in\mathfrak{g}\}$. Similarly, if $U\subset Z$ is a $G$-invariant open set, $\Omega^{\text{ch}}(U)^G$ coincides with $\Omega^{\text{ch}}(U)^{\gg}$, which is the joint kernel of $\{(L_{\xi}|_U)_0|\ \xi\in\mathfrak{g}\}$. If $U$ is an open set which is not $G$-invariant, $G$ does not act on $\Omega^{\text{ch}}(U)$. However, each $(L_{\xi}|_U)_0$ acts on $\Omega^{\text{ch}}(U)$ since for each point $x\in U$, there is a neighborhood $V$ of the identity element of $G$ such that for each $g\in V$, $g\cdot x \in U$. Therefore $\Omega^{\text{ch}}(U)^{\mathfrak{g}}$ is a well-defined vertex algebra.

\begin{lemma} \label{orbifoldsheaf} If $Z$ carries an action of a connected, compact Lie group $G$, the assignment $U \mapsto \Omega^{\text{ch}}(U)^{\mathfrak{g}}$ defines a sheaf of vertex algebras on $Z$.\end{lemma}

\begin{proof}
It is enough to show that being in the joint kernel of $\{(L_{\xi}|_U)_0|\ \xi\in\mathfrak{g}\}$ is a local condition. In other words, $a \in\Omega^{\text{ch}}(U)$ lies in this kernel if and only if for every $x\in U$, there is a neighborhood $V\ni x$ in $U$ such that $(L_{\xi}|_V)_0 (a|_V) = 0$. But this is clear from the fact that restriction maps are vertex algebra homomorphisms. \end{proof}

\subsection{Coordinate-free description} 
For any open set $U\subset Z$, recall that $f\in C^{\infty}(U)$ and $\omega \in \Omega^1(U)$ can be regarded as sections of $\Omega^{\text{ch}}(U)$ of weight zero and degrees $0$ and $1$, respectively. Additionally, given a vector field $X \in \text{Vect}(U)$, there are sections $$\iota_X(z) = \sum_{n\in \mathbb{Z}} (\iota_X)_n z^{-n-1},\qquad L_X(z) = \sum_{n\in \mathbb{Z}} (L_X)_n z^{-n-1}$$ in $\Omega^{\text{ch}}(U)$ of weight $1$ and degrees $-1$ and $0$, respectively. The local description of $\iota_X$ and $L_X$ appear in \cite{LL}. Let $\gamma^1,\dots, \gamma^n$ be local coordinates and $X = \sum_{i=1}^n f_i  \frac{\partial}{\partial \gamma^i}$ where $f_i = f_i (\gamma^1,\dots, \gamma^n)$ is a smooth function. Then $$ \iota_X = \sum_{i=1}^n :f_i b^i:,\qquad L_X = D(\iota_X) = \sum_{i=1}^n : \beta^i f_i: + \sum_{i=1}^n \sum_{j=1}^n : \frac{\partial f_j}{\partial \gamma^i} c^i b^j:.$$ The next theorem\footnote{The coordinate-free description of the relations is due to Bailin Song, and we thank him for sharing this result with us.} gives a useful coordinate-independent description of $\Omega^{\text{ch}}(U)$ when $U$ is a coordinate open set.

\begin{theorem}  For a coordinate open set $U\subset \mathbb{R}^n$, $\Omega^{\text{ch}}(U)$ is strongly generated by the following fields: \begin{equation} \label{coordinateindep} f \in  C^{\infty}(U),\qquad \omega \in \Omega^1(U),\qquad \{L_X, \iota_X|\ X\in \text{Vect}(U)\}.\end{equation} These satisfy the following OPE relations.
\begin{equation}\label{opecdr} \begin{split}
\iota_X(z) \iota_Y(w) \sim 0,\\
L_X(z) \iota_Y(w) \sim \iota_{[X,Y]}(w)(z-w)^{-1},\\
L_X(z) L_Y(w) \sim L_{[X,Y]}(w)(z-w)^{-1},\\
L_X(z) \omega(w) \sim \text{Lie}_X(\omega)(w)(z-w)^{-1},\\
\iota_X(z) \omega(w) \sim \iota_X(\omega)(w)(z-w)^{-1},\\
L_X(z) f(w) \sim X(f)(w)(z-w)^{-1},\\
\iota_X(z) f(w) \sim 0.
\end{split}\end{equation}
The ideal of relations among these generators is generated by the following elements.
\begin{multline}\label{cirelations} 
1-id,\qquad  :f g: - fg, \qquad  :\nu \omega: -  \nu\omega , \qquad f,g \in C^{\infty}(U), \qquad  \nu,\omega \in \Omega^1(U),\\
\iota_{gX}-: g\iota_X:,\qquad L_{gX}-: (dg) \iota_X:-:gL_X:, \qquad g \in C^{\infty}(U),\qquad X \in \text{Vect}(U),\\
\partial g(\phi_1,\dots,\phi_n)-\sum_{i=1}^n  \frac{\partial g}{\partial{\phi_i}} \partial \phi_i ,\qquad g\in C^{\infty}(\mathbb{R}^n),\qquad \phi_i \in C^{\infty}(U).
\end{multline}
\end{theorem}


\begin{proof} For a coordinate open set $U$ with coordinates $\gamma^i$ for $i=1,\dots, n$, \eqref{coordinateindep} is clearly a strong generating set for $\Omega^{\text{ch}}(U)$ since it contains the usual generators $f \in C^{\infty}(U), b^i, c^i, \beta^i$ as a subset. Similarly, the set of relations \eqref{cirelations} are all consequences of the set \eqref{cdrrelations}, which is a subset of \eqref{cirelations}.
\end{proof}

We shall call an open set $U\subset Z$ {\it small} if $\Omega^{\text{ch}}(U)$ has the strong generating set \eqref{coordinateindep}. We shall call an open cover $\{U_{\alpha}\}$ of $Z$ a {\it small open cover} if each $U_{\alpha}$ is small. Aside from coordinate open sets, there is one more type of small open set that will be important to us. These are of the form $U \times T^m$ where $U$ is a coordinate open set, and $T^m$ is a torus of rank $m$. The reason such a set is small is that if $y^1,\dots, y^m$ are coordinates on $T^m$ satisfying $y^i = y^i + 2\pi$, the corresponding fields $\partial y^i$, $c^i = d y^i$, $\beta^i$, and $b^i$, are globally defined. If $\pi: Z \ra M$ is a principal $S^1$-bundle, we shall often choose a trivializing open cover $\{V_{\alpha}\}$ for $M$ such that each $V_{\alpha}$ is a coordinate open set. Then $\{U_{\alpha} = \pi^{-1}(V_{\alpha})\}$ is a small open cover for $Z$, and each $U_{\alpha} \cong V_{\alpha} \times S^1$.

\begin{remark} The advantage of having a local but coordinate-independent description of a sheaf of vertex algebras on a manifold $M$ is the following principle. Given vertex algebra sheaves $\cA_M$ and $\cB_M$ on $M$, to specify a morphism $\cA_M \rightarrow \cB_M$, it is enough to give a vertex algebra homomorphism $\phi_{\alpha}: \cA(U_{\alpha}) \rightarrow \cB(U_{\alpha})$ such that $\phi_{\alpha}$ and $\phi_{\beta}$ agree on the overlap $U_{\alpha} \cap U_{\beta}$. If we have coordinate-independent generators and relations for $\cA_M$ and $\cB_M$, it is enough to show that OPEs among the generators are preserved and the ideal of relations is annihilated; the agreement on overlaps is then automatic. The same principle applies to morphisms of sheaves of modules over vertex algebras. \end{remark}

\begin{remark} If $Z$ has finite topological type, it is a theorem of Bailin Song \cite{Song} that {\it any} open set $U\subset Z$ is small. In particular, the global section algebra $\Omega^{\text{ch}}(Z)$ has the strong generating set \eqref{coordinateindep} and relations \eqref{cirelations}. \end{remark}

\subsection{Twisted chiral de Rham differential}
\begin{lemma}
Let $H\in \Omega^{2k+1}(Z)$ be a closed differential form of degree $2k+1$. We define the {\it twisted differential} $D_H$ on $\Omega^{\text{ch}}(Z)$ by $$D_H(\omega) = D(\omega) + :H\omega:,$$ which coincides with the classical twisted differential $d_H = d + H\wedge$ at weight zero. Then $D_H^2=0$.
\end{lemma}

\begin{proof}
$D_H$ is a square-zero operator since $D$ is a derivation, $H$ is $D$-closed, and also $:H (:H \omega:):\ = 0$ for any $\omega \in \Omega^{\text{ch}}(Z)$. 
\end{proof}

Since $H$ is odd and has weight zero, $D_H$ is odd and preserves conformal weight. Therefore it is homogeneous of degree $1$ with respect to the $\mathbb{Z}/2\mathbb{Z}$ grading obtaining by reducing the degree grading modulo $2$. We denote this $\mathbb{Z}/2\mathbb{Z}$-graded complex by $(\Omega^{\text{ch}, \bullet}(Z), D_H)$. Note that $D_H$ is not a vertex algebra derivation of $\Omega^{\text{ch}, \bullet}(Z)$, just as $d_H$ is not a derivation of $\Omega^{\bullet}(Z)$. However, if we regard $\Omega^{\text{ch}, \bullet}(Z)$ not as a vertex algebra but as a module over itself with generator the vacuum vector, then $D_H$ is a derivation in the category of modules. In other words, \begin{equation} \label{module derivation} D_H(a \circ_k b) = D(a) \circ_k b + (-1)^{|a|} a \circ_k D_H(b),\qquad k\in \mathbb{Z}.\end{equation} Moreover, $H^{\bullet}(\Omega^{\text{ch}}(Z), D_H)$ has a well-defined $\mathbb{Z}/2\mathbb{Z}$ grading, and at weight zero it coincides with $H^{\bullet}(Z,H) = H^{\bullet}(\Omega(Z), d_H)$.

\begin{theorem} \label{thm:vanishing} $H^{\bullet} (\Omega^{\text{ch}}(Z), D_H)$ vanishes in positive weight. Therefore the inclusion $(\Omega^{\bullet}(Z), d_H) \hookrightarrow (\Omega^{\text{ch}, \bullet}(Z), D_H)$ is a quasi-isomorphism.
\end{theorem}

\begin{proof} Let $A$ be a cohomology class of weight $w>0$ with respect to $D_H$. From the local description of $\Omega^{\text{ch}}(Z)$, it is clear that the weight-homogeneous subspaces of $\Omega^{\text{ch}}(Z)$ have bounded degrees both from above and below. Any representative of $A$ has the form $$a = \sum_{i\geq k} a^i$$ for some $k$, where $a^i$ has weight $w$ and degree $i$. We may choose our representative so that $k$ is as large as possible. By degree considerations, $a^k$ must be $D$-closed. Since $w>0$, $$b = \frac{1}{w} G_0(a^k) \in \Omega^{\text{ch}}(Z)$$ has degree $k-1$ and weight $w$, and satisfies $D(b) = a^k$. Note that $a - D_H(b)$ represents the class $A$, but the piece in degree $k$ has been canceled. Since $k$ was maximal, we must have $a - D_H(b) = 0$. \end{proof}

One of the difficult features of the chiral de Rham complex is that it is not functorial since it is built from both differential forms and vector fields. In the years after it was introduced, alternative formulations have been given in the language of Courant algebroids \cite{Bressler,BHS,Heluani1}, which explains the nonclassical transformation formula for $\beta^i$ appearing in \eqref{cdrlocaltransform}. From our point of view, regarding $\Omega^{\text{ch}}_Z$ as a structure naturally associated to a Courant algebroid endows it with enough functorial properties to formulate a version of T-duality. From here on, we assume that the degree of $H$ is equal to $3$.

\section{Courant algebroids and vertex algebra sheaves}\label{sect:courantsheaves}
An important result of Heluani (Proposition 4.6 of \cite{Heluani1}) is that for any Courant algebroid $E \ra Z$ there is an associated sheaf $\cU^E$ of vertex algebras on $Z$. In Heluani's formalism, $\cU^E$ is a sheaf of $N=1$ SUSY vertex algebras, but we shall only need the underlying ordinary vertex algebra structure. The bracket relations in $E$ give rise to a sheaf of Lie conformal algebras on $Z$, and $\cU^E$ is the quotient of the corresponding sheaf of universal enveloping vertex algebras by a certain sheaf of vertex algebra ideals. Locally, the generators and relations are the analogues of \eqref{coordinateindep} and \eqref{cirelations}. For any $E$, and any open set $U\subset Z$, there are inclusions $$i: C^{\infty}(U) \ra \cU^E(U),\qquad j: E \ra   \cU^E(U)$$ satisfying OPE relations coming from $\langle \cdot , \cdot \rangle$, $[\cdot ,\cdot]$, and $\pi$. If $E$ is the untwisted Courant algebroid $TZ \oplus T^*Z$, $\cU^E$ is just the chiral de Rham complex. The assignment $E \mapsto \cU^E$ is functorial in an obvious way: morphisms of Courant algebroids give rise to morphisms of the corresponding vertex algebra sheaves. For example, suppose that $\iota\colon S\longhookrightarrow Z$ is a totally geodesic inclusion of Riemannian manifolds, as in \S\ref{sect:courant}. The morphism $R_\iota: TZ\oplus T^*Z \longrightarrow TS\oplus T^*S$ of standard Courant algebroids give rise to a morphism $R^{\text{ch}}_\iota:\Omega^{\text{ch}}_Z \ra \iota_*(\Omega^{\text{ch}}_S)$ of vertex algebra sheaves on $Z$. In particular we have a homomorphism of global section algebras $\Omega^{\text{ch}}(Z) \ra \Omega^{\text{ch}}(S)$ which extends the restriction map of differential forms at weight zero.

\subsection{Exact Courant algebroids} Let $E$ be an exact Courant algebroid. Choose a splitting $E \cong TZ \oplus T^*Z$ and a closed $3$-form $H \in \Omega^3(Z,\mathbb{R})$ representing the cohomology class of $E$, so that the bracket is given by \eqref{twistedbracket}. Because of its close connection to $\Omega^{\text{ch}}_Z$, we shall denote the sheaf $\cU^E$ by $\Omega^{\text{ch}, H}_Z$, and we denote the algebra of sections over $U\subset Z$ by $\Omega^{\text{ch},H}(U)$. We call it the {\it $H$-twisted chiral de Rham complex} of $Z$. For any covering $\{U_{\alpha}\}$ of $Z$ by small open sets, $\Omega^{\text{ch},H}(U_{\alpha})$ is strongly generated by \begin{equation} \label{twistedgenerators} \tilde{f} \in C^\infty(U_{\alpha}), \qquad \tilde{\omega} \in\Omega^1(U_{\alpha}),\qquad  \{\tilde{L}_X, \tilde{\iota}_X|\ X\in \text{Vect}(U_{\alpha})\}.\end{equation} Here the tilde notation is used to distinguish these fields from the corresponding generators of $\Omega^{\text{ch}}(U_{\alpha})$. These generators satisfy the following OPEs, which are the twisted analogues of \eqref{opecdr}.
\begin{equation} \label{opecdrtwist}
\begin{split}
\tilde{\iota}_X(z) \tilde{\iota}_Y(w) \sim 0,\\
\tilde{L}_X(z) \tilde{\iota}_Y(w) \sim \big(\tilde{\iota}_{[X,Y]}(w) + (\widetilde{\iota_X \iota_Y H})(w)\big)(z-w)^{-1},\\
\tilde{L}_X(z) \tilde{L}_Y(w) \sim \big(\tilde{L}_{[X,Y]}(w) +( \widetilde{D \iota_X \iota_Y H})(w)\big)(z-w)^{-1},\\
\tilde{L}_X(z) \tilde{\omega}(w) \sim \widetilde{\text{Lie}_X(\omega)}(w)(z-w)^{-1},\\
\tilde{\iota}_X(z) \tilde{\omega}(w) \sim \widetilde{\iota_X(\omega)}(w)(z-w)^{-1},\\
\tilde{L}_X(z) \tilde{f}(w) \sim \widetilde{X(f)}(w)(z-w)^{-1},\\
\tilde{\iota}_X(z) \tilde{f}(w) \sim 0.
\end{split}
\end{equation}
The ideal of relations among these fields has the same generating set \eqref{cirelations} as the untwisted case, where each field is replaced by the tilde version. Even though the exact Courant algebroids $E$ are classified by $H^3(Z,\mathbb{R})$, it turns out that the corresponding vertex algebra sheaves are all isomorphic to the {\it untwisted} chiral de Rham sheaf.

\begin{theorem}[Untwisting trick] \label{untwisting} Let $\{U_{\alpha}\}$ be a small open cover of $Z$. Define a map $\Omega^{\text{ch}}(U_{\alpha})\ra \Omega^{\text{ch},H}(U_{\alpha})$ by \begin{equation} \label{eqn:untwisting} \iota_X \mapsto \tilde{\iota}_X,\qquad L_X \mapsto \tilde{L}_X - \widetilde{\iota_X H},\qquad f \mapsto \tilde{f},\qquad \omega \mapsto \tilde{\omega}.\end{equation}This is an isomorphism of vertex algebras which is well-defined on overlaps $U_{\alpha}\cap U_{\beta}$, so it is independent of the choice of cover and defines a sheaf isomorphism $\Omega^{\text{ch}}_Z \cong \Omega^{\text{ch},H}_Z$.
\end{theorem}

\begin{proof} For $X,Y\in \text{Vect}(U_{\alpha})$, we compute 
\begin{equation*}
\begin{split}
(\tilde{L}_X - \widetilde{\iota_X H})(z) \tilde{\iota}_Y(w) \sim \big(\tilde{\iota}_{[X,Y]} + \widetilde{\iota_X \iota_Y H} - (\widetilde{\iota_X H}) \circ_0 \tilde{\iota}_Y\big)(w)(z-w)^{-1} = \tilde{\iota}_{[X,Y]}(w)(z-w)^{-1}, \\
(\tilde{L}_X - \widetilde{\iota_X H})(z) (\tilde{L}_Y- \widetilde{\iota_Y H})(w) \sim \big( \tilde{L}_{[X,Y]} +\widetilde{D(\iota_X \iota_Y H})  - \widetilde{\text{Lie}_X \iota_Y H} + \widetilde{\text{Lie}_Y  \iota_X H} \big)(w) (z-w)^{-1} \\
= \big(\tilde{L}_{[X,Y]} + \widetilde{\text{Lie}_X \iota_Y H} - \widetilde{\iota_X \text{Lie}_Y H} - \widetilde{\text{Lie}_X \iota_Y H} + \widetilde{\text{Lie}_Y  \iota_X H}\big)(w)(z-w)^{-1}\\ = \big(\tilde{L}_{[X,Y]} - \widetilde{\iota_X \text{Lie}_Y H} + \widetilde{\text{Lie}_Y  \iota_X H}\big)(w)(z-w)^{-1} \\ = \big(\tilde{L}_{[X,Y]}- \widetilde{\iota_{[X,Y]}H}\big)(w)(z-w)^{-1}.
 \end{split}
 \end{equation*} It is easy to check that \eqref{eqn:untwisting} preserves the remaining OPE relations. The fact that it annihilates the ideal of relations is also clear, as is the bijectivity.
\end{proof}

As in the untwisted case, for a coordinate open set with coordinates $\tilde{\gamma}^1, \dots, \tilde{\gamma}^n$, the following locally defined fields generate a topological vertex algebra of rank $n$. 
$$\tilde{J} = \sum_{i=1}^n :\tilde{b}^i \tilde{c}^i:,\qquad \tilde{Q} = \sum_{i=1}^n :(\tilde{\beta}^i - \widetilde{\iota_i H}) \tilde{c}^i:,\qquad \tilde{G} = \sum_{i=1}^n :\tilde{b}^i \partial \tilde{\gamma}^i:, $$ $$ \tilde{L} = \sum_{i=1}^n :(\tilde{\beta}^i - \widetilde{\iota_i H})\partial \tilde{\gamma}^i: - :\tilde{b}^i \partial \tilde{c}^i:.$$
The fields $\tilde{L},\tilde{G}$ are globally defined, and the modes $\tilde{J}_0, \tilde{Q}_0$ are globally defined. Define the differential $\tilde{D}$ to be $\tilde{Q}_0$, which satisfies $$\tilde{D}(\tilde{f}) = \widetilde{df},\qquad \tilde{D}(\tilde{\omega}) = \widetilde{d\omega},\qquad \tilde{D}(\tilde{\iota}_X) =  \tilde{L}_X - \widetilde{\iota_X H},\qquad \tilde{D}(\tilde{L}_X - \widetilde{\iota_XH}) = 0.$$ 
As in the untwisted setting, $\tilde{L}(z)$ is a Virasoro element with central charge zero, $\tilde{f}$, $\tilde{\omega}$, $\tilde{\iota}_X$, $\tilde{L}_X - \widetilde{\iota_X H}$ are primary of weights $0,0,1,1$ with respect to $\tilde{L}(z)$, and we have $[\tilde{D}, \tilde{G}_0] = \tilde{L}_0$. It follows that $H^*(\Omega^{\text{ch},H}(Z), \tilde{D})$ vanishes in positive weight and coincides with $H^*(Z)$ at weight zero. The map \eqref{eqn:untwisting} intertwines the differentials $D$ and $\tilde{D}$, so it induces an isomorphism of differential graded vertex algebra sheaves \begin{equation} \label{eqn:untwisting2} (\Omega^{\text{ch}}_Z, D) \cong (\Omega^{\text{ch},H}_Z, \tilde{D}).\end{equation} 
Finally, we may regard $\Omega^{\text{ch},H}_Z$ as a sheaf of modules over itself, and we have the {\it twisted} differential $\tilde{D}_H$ defined by $\tilde{D}_H(\omega) = \tilde{D}(\omega) + :\tilde{H} \omega:$. It is homogeneous of degree $1$ with respect to the $\mathbb{Z}/2\mathbb{Z}$-grading, and we denote this complex by $(\Omega^{\text{ch},H,\bullet}_Z, \tilde{D}_H)$. As in the untwisted case, $H^{\bullet}(\Omega^{\text{ch},H}(Z), \tilde{D}_H)$ vanishes in positive weight. Finally, we have an isomorphism of differential graded sheaves of modules \begin{equation} \label{eqn:untwisting3} (\Omega^{\text{ch},\bullet}_Z, D_H) \cong (\Omega^{\text{ch},H,\bullet}_Z, \tilde{D}_H).\end{equation}

\subsection{Invariant Courant algebroids from principal circle bundles} Suppose that $\pi: Z \ra M$ is a principal circle bundle. Let $E$ denote the quotient $(TZ \oplus T^*Z)/{S^1}$ of the standard Courant algebroid on $Z$ by the $S^1$-action. It is well known that $E$ is a Courant algebroid which is not exact, but is transitive; that is, the anchor map is surjective. As usual, fix a connection form $A \in \Omega^1(Z)$, and let $X_A$ denote the vector field on $Z$ generated by the $S^1$-action, normalized so that $\iota_{A}A = 1$. Recall that in the dimension reduction formalism, we regard $E$ as a Courant algebroid on $M$, and we identify vector fields in $\text{Vect}_{\text{hor}}(Z)$ with vector fields on $M$. Here $\text{Vect}_{\text{hor}}(Z)$ consists of $S^1$-invariant vector fields on $Z$ such that $\iota_X A = 0$. 


Let $L_A, \iota_A \in \Omega^{\text{ch}}(Z)$ denote the global sections corresponding to $X_A$ of weight one and degrees $0$ and $-1$, respectively. In this section, we shall consider three vertex algebra sheaves on $Z$: the $i\mathbb{R}$-invariant sheaf $(\Omega^{\text{ch}}_Z)^{i\mathbb{R}}$, the $i\mathbb{R}[t]$-invariant sheaf $(\Omega^{\text{ch}}_Z)^{i\mathbb{R}[t]}$, and the quotient sheaf $(\Omega^{\text{ch}}_Z)^{i\mathbb{R}[t]} / \langle L_A \rangle$. At weight zero, $\Omega^{\text{ch}}(Z)^{i\mathbb{R}}$, $\Omega^{\text{ch}}(Z)^{i\mathbb{R}[t]}$, and $\Omega^{\text{ch}}(Z)^{i\mathbb{R}[t]} / \langle L_A \rangle$ all coincide with $\Omega(Z)^{S^1}$. The sheaf $(\Omega^{\text{ch}}_Z)^{i\mathbb{R}[t]} / \langle L_A \rangle$ will be the most important for our purposes, and is closely related to the sheaf $\cU^E$ associated to $E$ by Heluani's result. The differential $D$ restricts to both $\Omega^{\text{ch}}(Z)^{i\mathbb{R}}$ and $\Omega^{\text{ch}}(Z)^{i\mathbb{R}[t]}$, and it descends to a differential on $\Omega^{\text{ch}}(Z)^{i\mathbb{R}[t]} / \langle L_A \rangle$. We will compute the $D$-cohomology of these global section algebras when $Z$ has finite topological type. Finally, we will consider the $H$-twisted versions of these sheaves, and show that they are isomorphic to the untwisted versions.



First, the zero mode $(L_A)_0$ infinitesimally generates the $S^1$-action on $\Omega^{\text{ch}}(Z)$. Identifying the Lie algebra of $S^1$ with $i\mathbb{R}$, for any open set $U\subset Z$, the kernel of $(L_A)_0$ in $\Omega^{\text{ch}}(U)$ coincides with $\Omega^{\text{ch}}(U)^{i \mathbb{R}}$. By Lemma \ref{orbifoldsheaf}, \begin{equation} \label{invariant sheaf} U\mapsto \Omega^{\text{ch}}(U)^{i \mathbb{R}}\end{equation} defines a sheaf of vertex algebra on $Z$, which we denote by $\big(\Omega^{\text{ch}}_Z\big)^{i \mathbb{R}}$. Note that $\Omega^{\text{ch}}(U)^{i \mathbb{R}} = \Omega^{\text{ch}}(U)^{S^1}$ if $U$ is an $S^1$-invariant open set; in particular $\Omega^{\text{ch}}(Z)^{i \mathbb{R}} = \Omega^{\text{ch}}(Z)^{S^1}$.




The condition $\iota_{A}A = 1$ is equivalent to the OPE relation 
\begin{equation} \label{copyofbc} \iota_A(z) A(w) \sim (z-w)^{-1},\end{equation} 
so the odd fields $\iota_A, A$ generate a copy of the rank one $bc$-system $\cE$. Recall that $dA$ is just the curvature form $F_A$, which lies in $\pi^*(\Omega^2(M))$. Let $\Omega^{\text{ch}}_0(M)$ denote the subcomplex of $\Omega^{\text{ch}}(M)$ generated by $\Omega^*(M)$, regarded as the weight zero subspace. By Theorem 3.6 of \cite{LLS}, $\Omega^{\text{ch}}_0(-)$ is contravariantly functorial, so $\pi^*(\Omega^{\text{ch}}_0(M))$ is a well-defined subalgebra of $\Omega^{\text{ch}}_0(Z) \subset \Omega^{\text{ch}}(Z)$. In fact, this subalgebra is preserved by both $D$ and $G_0$. 
Note that $\partial (dA) = D \partial A \in \pi^*(\Omega^{\text{ch}}_0(M))$ has weight one and is $D$-closed, and hence is $D$-exact. Therefore we can find $\xi^A \in \pi^*(\Omega^{\text{ch}}_0(M))$ of weight one and degree one, satisfying $D \xi^A = \partial dA$. Define \begin{equation} \label{defofgamma} \Gamma^A = G_0 (\partial A),\end{equation} which is a global section of $\Omega^{\text{ch}}_0(Z)$ of weight one and degree zero, and satisfies $D \Gamma^A = \partial A - \xi^A$.


Fix a trivializing open cover $\{V_{\alpha}\}$ for $M$ such that each $V_{\alpha}$ is a coordinate open set. Then $\{U_{\alpha} = \pi^{-1}(V_{\alpha})\}$ is a small open cover for $Z$, and $U_{\alpha} \cong V_{\alpha} \times S^1$. Using the splittings of $(TU_{\alpha})/S^1$ and $(T^*U_{\alpha})/S^1$ induced by the connection form, we obtain the following strong generating set for $\Omega^{\text{ch}}(U_{\alpha})^{i \mathbb{R}}$:
\begin{equation} \label{invariantgenerators} f \in \pi^*(C^{\infty}(V_{\alpha})),\qquad \omega \in \pi^*(\Omega^1(V_{\alpha})),\qquad A,\iota_A, L_A,\Gamma^A, \qquad \{\iota_X, L_X|\ X \in \text{Vect}_{\text{hor}}(U_{\alpha})\}.\end{equation} Recall that $\text{Vect}_{\text{hor}}(U_{\alpha})$ consists of $S^1$-invariant vector fields $X$ on $U_{\alpha}$ such that $\iota_X A = 0$, and we may identify $\text{Vect}_{\text{hor}}(U_{\alpha})$ with $\text{Vect}(V_{\alpha})$.

\begin{lemma} \label{heisenberg} The fields $L_A$ and $\Gamma^A$ satisfy OPE relations
\begin{equation} \label{copyofheis} L_A(z) \Gamma^A(w) \sim (z-w)^{-2},\qquad  L_A(z)L_A(w) \sim 0,\qquad \Gamma^A(z)\Gamma^A(w)\sim 0,\end{equation} so they generate a copy of the Heisenberg algebra $\cH$ of rank $2$.
\end{lemma}

\begin{proof} It follows from \eqref{copyofbc} that $\iota_A \circ_1 \partial A =1$. Using $D \iota_A = L_A$, we have
$$L_A\circ_1 \Gamma^A = L_A\circ_1 (G_0 \partial A) = (D \iota_A) \circ_ 1  (G_0 \partial A) = D(\iota_A \circ_1 (G_0\partial A))  + \iota_A\circ_1 D (G_0 \partial A)$$
$$ = D(\iota_A \circ_1 (G_0\partial A))  + \iota_A\circ_1 (L_0 \partial A) - \iota_A \circ_1 G_0 (D\partial A)).$$
Since $G_0\partial A$ has weight one and degree zero, it cannot depend on $\partial A$: any monomial involving $\partial A$ must also depend on some $\iota_X$ or $\iota_A$ by degree considerations, but then would have weight at least two. Therefore $\iota_A \circ_1 (G_0 \partial A) = 0$. Likewise, since $dA \in \pi^*(\Omega^2(M))$, $D\partial A$ lies in the subalgebra $\pi^*(\Omega^{\text{ch}}_0(M)) \subset \Omega^{\text{ch}}(Z)$, which is preserved by $G_0$. This space is {\it chiral basic} in the terminology of \cite{LLS}, meaning that it commutes with both $L_A$ and $\iota_A$. Therefore $\iota_A \circ_1 G_0 (D\partial A))=0$, so we conclude that $$L_A\circ_1 (\Gamma^A) =  \iota_A \circ_1 (L_0 \partial A) = \iota_A \circ_1 \partial A = 1.$$ The fact that $L_A \circ_k \Gamma^A = 0$ for $k>1$ is clear by weight considerations, and the fact that $L_A \circ_0 \Gamma^A = 0$ is clear because $\Gamma^A$ is $S^1$-invariant. The remaining OPE relations in the lemma are obvious. \end{proof}

\begin{lemma} \label{invariantvanishing} $H^*(\Omega^{\text{ch}}(Z))^{i\mathbb{R}}, D)$ vanishes in positive weight. Therefore the inclusion $\Omega^*(Z)^{S^1} \hookrightarrow \Omega^{\text{ch}}(Z)^{i\mathbb{R}}$ induces an isomorphism in cohomology. \end{lemma}

\begin{proof} Since $[D,G_0] = L_0$, it suffices to show that $G_0$ preserves $\Omega^{\text{ch}}(Z)^{i\mathbb{R}}$. But this is clear because $G_0$ preserves $\pi^*(\Omega^{\text{ch}}_0(M))$, and $G_0(L_X) = \iota_X$, $G_0(\partial A) = \Gamma^A$, $G_0(\Gamma^A) = 0$, and $G_0(L_A) = \iota_A$.\end{proof}

We may also consider $(\Omega^{\text{ch}}_Z)^{i\mathbb{R}}$ as a sheaf of $\mathbb{Z}/2\mathbb{Z}$-graded modules over itself. We use the notation $(\Omega^{\text{ch}, \bullet}_Z)^{i\mathbb{R}}$, and we denote the cohomology of its global section algebra by $H^{\bullet} (\Omega^{\text{ch}}(Z)^{i\mathbb{R}}, D_H)$. By a similar calculation, $H^{\bullet} (\Omega^{\text{ch}}(Z)^{i\mathbb{R}}, D_H)$ vanishes in positive weight and coincides with $H^{\bullet}(Z,H)$ in weight zero.

Next, we consider another vertex algebra sheaf on $Z$, namely, the commutant sheaf $$U \mapsto \text{Com}(L_A|_U, \Omega^{\text{ch}}(U)).$$ For any open set $U \subset Z$, $\text{Com}(L_A|_U, \Omega^{\text{ch}}(U))$ is the joint kernel of $\{(L_A|_U)_k|\ k\geq 0\}$, which generates the Lie algebra $i\mathbb{R}[t]$. Therefore $\text{Com}(L_A|_{U_{\alpha}}, \Omega^{\text{ch}}(U_{\alpha}))$ is just the invariant space $\Omega^{\text{ch}}(U_{\alpha})^{i \mathbb{R}[t]}$, which is properly contained in $\Omega^{\text{ch}}(U_{\alpha})^{i\mathbb{R}}$. As above, it not difficult to see that $\Omega^{\text{ch}}(U_{\alpha})^{i \mathbb{R}[t]}$ has strong generators
\begin{equation} \label{heluanigenerators} f \in \pi^*(C^{\infty}(V_{\alpha})),\qquad \omega \in \pi^*(\Omega^1(V_{\alpha})),\qquad A,\iota_A, L_A, \qquad \{\iota_X, L_X|\ X \in \text{Vect}_{\text{hor}}(U_{\alpha})\}.\end{equation} 
In particular, $\Gamma^A \notin \Omega^{\text{ch}}(U_{\alpha})^{i \mathbb{R}[t]}$ since it does not commute with $L_A$. 

If $Z$ has finite topological type, recall that $Z$ is itself a small open set. Then \eqref{heluanigenerators} with $U_{\alpha}$ replaced by $Z$ and $V_{\alpha}$ replaced by $M$, is a strong generating set for $\Omega^{\text{ch}}(Z)^{i \mathbb{R}[t]}$.

\begin{theorem} \label{cohomology-comm} Suppose that $Z$ has finite topological type. Then $H^*(\Omega^{\text{ch}}(Z)^{i\mathbb{R}[t]}, D)$ is isomorphic to $H^*(Z) \otimes \cJ$. Here $\cJ$ is the free odd abelian vertex algebra with generator $[\alpha^A]$ of degree one and weight one, where $\alpha^A = \partial A - \xi^A$. \end{theorem}

\begin{proof} Since $D(\Gamma^A) = \alpha^A $ in $\Omega^{\text{ch}}(Z)$ we have $D(\alpha^A) = 0$. The same argument as the proof of Lemma \ref{heisenberg} shows that if $\Gamma \in \Omega^{\text{ch}}(Z)$ is any element satisfying $D(\Gamma) =\alpha^A$, then $L_A \circ_1 \Gamma = 1$. Therefore $\Gamma$ cannot lie in $\Omega^{\text{ch}}(Z)^{i\mathbb{R}[t]} = \text{Com}(L_A, \Omega^{\text{ch}}(Z))$, so $\alpha^A$ represents a nontrivial cohomology class. 

To see that the cohomology algebra has the above description, we consider a filtration $$\Omega^{\text{ch}}(Z)^{i\mathbb{R}[t]}_{(0)} \subset \Omega^{\text{ch}}(Z)^{i\mathbb{R}[t]}_{(1)} \subset \cdots$$ on $\Omega^{\text{ch}}(Z)^{i\mathbb{R}[t]}$ where $\Omega^{\text{ch}}(Z)^{i\mathbb{R}[t]}_{(d)}$ is spanned by normally ordered monomials in the generators $f, \omega, A, \iota_A, \iota_X, L_A, L_X$ and their derivatives, such that at most $d$ of the fields $\iota_A, \iota_X, L_A, L_X$ and their derivatives appear. It follows from the OPE relations \eqref{opecdr} that this is a good increasing filtration in the sense of \cite{LiII}. Setting $\Omega^{\text{ch}}(Z)^{i\mathbb{R}[t]}_{(-1)} = 0$, the associated graded algebra $$\text{gr}(\Omega^{\text{ch}}(Z)^{i\mathbb{R}[t]}) = \bigoplus_{j\geq 0} \Omega^{\text{ch}}(Z)^{i\mathbb{R}[t]}_{(j)} / \Omega^{\text{ch}}(Z)^{i\mathbb{R}[t]}_{(j-1)}$$ is then an associative, supercommutative algebra on the same generating set. Both $D$ and $G_0$ preserve the filtration, and hence act on $\text{gr}(\Omega^{\text{ch}}(Z)^{i\mathbb{R}[t]})$ by derivations. We have $$\text{gr}(\Omega^{\text{ch}}(Z)^{i\mathbb{R}[t]}) \cong \cG \otimes \cJ,$$ where $\cG$ is the algebra generated by $f, \omega, \iota_A, \iota_X, L_A, L_X$ and their derivatives, together with $A$ (but not its derivatives), and $\cJ$ is generated by $\alpha^A$ and its derivatives. Note that $D$ and $G_0$ act on $\cG$, so that $H^*(\cG,D)$ vanishes in positive weight and coincides with $H^*(Z)$ at weight zero. Also, $D$ acts trivially on $\cJ$, so $H^*(\cJ, D) \cong \cJ$. Therefore $H^*(\text{gr}(\Omega^{\text{ch}}(Z)^{i\mathbb{R}[t]}))\cong H^*(Z) \otimes \cJ$. The result then follows by induction on filtration degree.
\end{proof}


\begin{corollary} \label{twistedcohomology-comm} Suppose that $Z$ has finite topological type. Then $$H^{\bullet}(\Omega^{\text{ch}}(Z)^{i\mathbb{R}[t]}, D_H) \cong H^{\bullet}(Z,H) \otimes \cJ.$$
\end{corollary} 
\begin{proof} 
Given a contractible open set $U\subset Z$, note first that we can find some $B \in \Omega^2(U)$ such that $H|_U = dB$. Since the map $e^B: \Omega^{\bullet} (U) \rightarrow \Omega^{\bullet} (U)$ satisfies $e^{-B} d e^B = d_H$, it is an isomorphism of complexes $(\Omega^{\bullet}(U), d_H) \cong (\Omega^{\bullet}(U), d)$, and it induces an isomorphism $H^{\bullet}(\Omega(U),d_H) \cong H^{\bullet}(\Omega(U), d)$. Using the isomorphism $H^*(\Omega^{\text{ch}}(U)^{i\mathbb{R}[t]}, D) \cong H^*(\Omega(U),d) \otimes \cJ$, it follows that $$H^{\bullet}(\Omega^{\text{ch}}(U)^{i\mathbb{R}[t]}, D_H) \cong H^{\bullet}(\Omega(U), d_H) \otimes \cJ.$$ Finally, since $Z$ admits a finite cover by contractible open sets, a Mayer-Vietoris argument yields $H^{\bullet}(\Omega^{\text{ch}}(Z)^{i\mathbb{R}[t]}, D_H) \cong H^{\bullet}(Z,H) \otimes \cJ$.
\end{proof}

There is one more vertex algebra sheaf on $Z$ that we need to consider. Since $L_A$ commutes with itself, for any open set $U\subset Z$, $L_A|_U$ is central in $\Omega^{\text{ch}}(U)^{i\mathbb{R}[t]}$. The quotient $\Omega^{\text{ch}}(U_{\alpha})^{i\mathbb{R}[t]} / \langle L_A \rangle$ by the ideal generated by $L_A$ therefore has strong generators 
\begin{equation} \label{quotientgenerators} f \in \pi^*(C^{\infty}(V_{\alpha})),\qquad \omega \in \pi^*(\Omega^1(V_{\alpha})),\qquad A,\iota_A, \qquad \{\iota_X, L_X|\ X \in \text{Vect}_{\text{hor}}(U_{\alpha})\}.\end{equation} 
The assignment $U \mapsto \Omega^{\text{ch}}(U)^{i\mathbb{R}[t]}/ \langle L_A\rangle$ defines a sheaf of vertex algebras on $Z$, which we shall denote by $(\Omega^{\text{ch}}_Z)^{i \mathbb{R}[t]} / \langle L_A  \rangle$. The differential $D$ descends to a differential on this quotient which we also denote by $D$; note that $D \iota_A = 0$. We remark that $L_A$ is {\it not} central in $\Omega^{\text{ch}}(U)^{i\mathbb{R}}$; in fact, $\Omega^{\text{ch}}(U)^{i\mathbb{R}} / \langle L_A \rangle$ is trivial due to \eqref{copyofheis}. If $Z$ has finite topological type, \eqref{quotientgenerators} is a strong generating set for $\Omega^{\text{ch}}(Z)^{i \mathbb{R}[t]} / \langle L_A  \rangle$ with $U_{\alpha}$ replaced by $Z$ and $V_{\alpha}$ replaced by $M$.



The {\it symplectic fermion algebra} $\cF$ is a simple vertex superalgebra with odd generators  $\phi, \psi $ of weight one, satisfying
$$\phi(z) \psi (w) \sim (z-w)^{-2},\qquad \phi(z) \phi (w) \sim 0,\qquad \psi (z) \psi(w) \sim 0.$$
\begin{theorem} \label{cohomology-quot} If $Z$ has finite topological type, $H^*((\Omega^{\text{ch}}(Z)^{i\mathbb{R}[t]} / \langle L_A \rangle ), D)$ is isomorphic to $H^*(Z) \otimes \cF$. \end{theorem}

\begin{proof} First, $\iota_A$ represents a nontrivial cohomology class because $D(\iota_A) = 0$ but any preimage under $D$ would have degree $-2$ and weight one, which is impossible. As above $\alpha^A = \partial A - \xi^A$ also represents a nontrivial class, and it is easy to check that $$\phi \mapsto \iota_A, \qquad \psi \mapsto \alpha^A $$ defines a nontrivial homomorphism $\cF \ra H^*((\Omega^{\text{ch}}(Z)^{i\mathbb{R}[t]} / \langle L_A \rangle ), D)$ which is injective because $\cF$ is simple.

As with $\Omega^{\text{ch}}(Z)^{i\mathbb{R}[t]}$, there is a good increasing filtration on $\Omega^{\text{ch}}(Z)^{i\mathbb{R}[t]} / \langle L_A \rangle$ where $\Omega^{\text{ch}}(Z)^{i\mathbb{R}[t]} / \langle L_A \rangle_{(d)}$ is spanned by normally ordered monomials in the generators $f, \omega, A, \iota_A, \iota_X, L_X$ of degree at most $d$ in $\iota_A, \iota_X, L_X$ and their derivatives. Then $\text{gr}(\Omega^{\text{ch}}(Z)^{i\mathbb{R}[t]} / \langle L_A \rangle)$ is an associative, supercommutative algebra with the same generators. We write
$$\text{gr}(\Omega^{\text{ch}}(Z)^{i\mathbb{R}[t]} / \langle L_A \rangle) \cong \cG_1 \otimes \cG_2,$$ where $\cG_1$ is generated by $f, \omega, \iota_X, L_X$ and their derivatives together with $A$ (but not its derivatives), and $\cG_2$ is generated by $\iota_A, \alpha^A$ and their derivatives. Then $D$ and $G_0$ act on $\cG_1$, so $H^*(\cG_1, D) \cong H^*(Z)$. Note that $\cG_2 = \text{gr}(\cF)$ and $D$ acts trivially on $\cG_2$, so $H^*(\cG_2,D) \cong \cG_2$. Therefore $H^*(\text{gr}(\Omega^{\text{ch}}(Z)^{i\mathbb{R}[t]} / \langle L_A \rangle),D) \cong H^*(Z) \otimes \text{gr}(\cF)$. The result follows by induction on filtration degree. \end{proof}

\begin{corollary} \label{twistedcohomology-quot} If $Z$ has finite topological type, $$H^{\bullet}((\Omega^{\text{ch}}(Z)^{i\mathbb{R}[t]} / \langle L_A \rangle ), D_H) \cong H^{\bullet}(Z,H) \otimes \cF.$$
\end{corollary}

\begin{proof} This is the same as the proof of Corollary \ref{twistedcohomology-comm}.
\end{proof}

\begin{remark} Even though $H^*(\Omega^{\text{ch}}(Z)^{i\mathbb{R}[t]}, D)$ and $H^*((\Omega^{\text{ch}}(Z)^{i\mathbb{R}[t]} / \langle L_A \rangle), D)$ do not vanish in positive weight (and similarly for the $D_H$-cohomology), Theorems \ref{cohomology-comm} and \ref{cohomology-quot} show that the higher weight components only depend on fixed vertex algebras $\cJ$ and $\cF$, and do not carry any topological information beyond the ordinary cohomology of $Z$.
\end{remark}

Let $\cA$ be a vertex algebra with a bigrading $\cA = \bigoplus_{n\geq 0} \bigoplus_{d \in \mathbb{Z}} \cA^d[n]$, where $\cA^n[d]$ is the subspace of degree $d$ and weight $n$. If $\text{dim}( \cA^d[n])$ is finite for all $n,d$, we may define the {\it graded character} $$\chi(\cA; q,z) = \sum_{n\geq 0} \sum_{d\in \mathbb{Z}} \text{dim} (\cA^d[n]) q^n z^d.$$ Clearly $\chi(\cF; q,z) = \prod_{n\geq 1} \big(1 + q^n z\big)\big(1 + q^n z^{-1}\big)$ if $\phi$ and $\psi$ are assigned degrees $-1$ and $1$, respectively. If $Z$ has finite topological type, $\text{dim}(H^d(Z))$ is finite for all $d$ and we write $\chi(Z;z) = \sum_d \text{dim}(H^d(Z)) z^d$.

\begin{corollary} \label{character-quot} If $Z$ has finite topological type, we have $$\chi(H^*((\Omega^{\text{ch}}(Z)^{i\mathbb{R}[t]} / \langle L_A \rangle), D); q,z)= \chi(Z;z) \prod_{n\geq 1} \big(1 + q^n z\big)\big(1 + q^n z^{-1}\big).$$
\end{corollary}

There is a similar notion of graded character when the degree grading is by $\mathbb{Z}/2\mathbb{Z}$. The character of $H^{\bullet}((\Omega^{\text{ch}}(Z)^{i\mathbb{R}[t]} / \langle L_A \rangle), D_H)$ can be written in the same way whenever $Z$ has finite topological type.


Next, given an $S^1$-invariant closed $3$-form $H$ on $Z$, we consider the $H$-twisted versions of the sheaves $(\Omega^{\text{ch}}_Z)^{i\mathbb{R}}$, $(\Omega^{\text{ch}}_Z)^{i\mathbb{R}[t]}$, and $(\Omega^{\text{ch}}_Z)^{i\mathbb{R}[t]} / \langle L_A \rangle$. The $S^1$-action on $\Omega^{\text{ch},H}(U)$ is now infinitesimally generated by the zero mode of $\tilde{L}_A - \widetilde{\iota_A H}$. Then \begin{equation} \label{twistedinvariantsheaf} U\mapsto \Omega^{\text{ch},H}(U)^{i\mathbb{R}}\end{equation} defines a sheaf of vertex algebra on $Z$  which we denote by $(\Omega^{\text{ch},H}_Z)^{i\mathbb{R}}$. As above, we fix a trivializing open cover $\{V_{\alpha}\}$ for $M$ such that each $V_{\alpha}$ is a coordinate open set, so $\{U_{\alpha} = \pi^{-1}(V_{\alpha})\}$ is a small open cover for $Z$, and $U_{\alpha} \cong V_{\alpha} \times S^1$. Then $\Omega^{\text{ch},H}(U_{\alpha})^{i\mathbb{R}}$ has the following strong generating set:  \begin{equation} \label{twistedinvariantgenerators} \tilde{f} \in \pi^*(C^{\infty}(V_{\alpha})),\qquad \tilde{\omega} \in \pi^*(\Omega^1(V_{\alpha})),\qquad \tilde{A}, \tilde{\iota}_A, \tilde{L}_A,\tilde{\Gamma}^A, \qquad \{\tilde{\iota}_X, \tilde{L}_X|\ X \in \text{Vect}_{\text{hor}}(U_{\alpha})\}.\end{equation} As above, $\tilde{A}, \tilde{\iota_A}$ generate $\cE$, and $\tilde{L}_A - \widetilde{\iota_A H}, \tilde{\Gamma}^A$ generate $\cH$. 


Similarly, $$U \mapsto \text{Com}((\tilde{L}_A -  \widetilde{\iota_A H})|_{U}, \Omega^{\text{ch},H}(U)) = \Omega^{\text{ch},H}(U)^{i \mathbb{R}[t]}$$ defines a vertex algebra sheaf on $Z$ which we denote by $(\Omega^{\text{ch},H}_Z)^{i\mathbb{R}[t]}$. For each $U\subset Z$, $(\tilde{L}_A - \widetilde{\iota_A H})|_U$ is central in $ \Omega^{\text{ch},H}(U)^{i\mathbb{R}[t]}$, and $$U \mapsto \Omega^{\text{ch},H}(U)^{i \mathbb{R}[t]} / \langle \tilde{L}_A - \widetilde{\iota_A H} \rangle$$ defines a vertex algebra sheaf on $Z$, denoted by $(\Omega^{\text{ch},H}_Z)^{i\mathbb{R}[t]} / \langle \tilde{L}_A - \widetilde{\iota_A H} \rangle$. The differential $\tilde{D}$ descends to a differential $\tilde{D}$ on $(\Omega^{\text{ch},H}_Z)^{i\mathbb{R}[t]} / \langle \tilde{L}_A - \widetilde{\iota_A H} \rangle$, and $\tilde{D} \tilde{\iota}_A = 0$.


\begin{theorem} The restriction of \eqref{eqn:untwisting2} to the $i\mathbb{R}$-invariant and $i\mathbb{R}[t]$-invariant subsheaves induces isomorphisms of differential graded vertex algebra sheaves \begin{equation} \label{eqn:untwisting2chiralinvt}\big(\big(\Omega^{\text{ch}}_Z\big)^{i\mathbb{R}}, D\big) \cong \big(\big(\Omega^{\text{ch},H}_Z\big)^{i\mathbb{R}}, \tilde{D}\big),\qquad \big(\big(\Omega^{\text{ch}}_Z\big)^{i\mathbb{R}[t]}, D\big) \cong \big(\big(\Omega^{\text{ch},H}_Z\big)^{i\mathbb{R}[t]}, \tilde{D}\big).\end{equation} Since $L_A \mapsto \tilde{L}_A -\widetilde{\iota_A H}$ under \eqref{eqn:untwisting2}, there is an induced isomorphism of differential graded vertex algebra sheaves
\begin{equation} \label{eqn:untwisting2quot}\big(\big(\big(\Omega^{\text{ch}}_Z\big)^{i\mathbb{R}[t]}/ \langle L_A \rangle\big), D\big) \cong\big(\big(\big(\Omega^{\text{ch},H}_Z\big)^{i\mathbb{R}[t]} / \langle \tilde{L}_A -\widetilde{\iota_A H}\rangle\big), \tilde{D}\big).\end{equation}
There are similar isomorphisms of differential graded sheaves of modules \begin{equation} \big(\big(\Omega^{\text{ch},\bullet }_Z\big)^{i\mathbb{R}}, D_H\big) \cong \big(\big(\Omega^{\text{ch},H,\bullet}_Z\big)^{i\mathbb{R}}, \tilde{D}_H\big),\end{equation}
\begin{equation}\big(\big(\Omega^{\text{ch},\bullet }_Z\big)^{i\mathbb{R}[t]}, D_H\big) \cong \big(\big(\Omega^{\text{ch},H,\bullet}_Z\big)^{i\mathbb{R}[t]}, \tilde{D}_H\big),\end{equation}
\begin{equation}\label{eqn:untwisting2quotmod}\big(\big(\big(\Omega^{\text{ch}, \bullet}_Z\big)^{i\mathbb{R}[t]}/ \langle L_A \rangle\big), D_H\big) \cong \big(\big(\big(\Omega^{\text{ch},H,\bullet}_Z\big)^{i\mathbb{R}[t]} / \langle \tilde{L}_A -\widetilde{\iota_A H}\rangle,\big) \tilde{D}_H\big).\end{equation}
\end{theorem}

\begin{proof} This is an immediate consequence of Theorem \ref{untwisting}.
\end{proof}

\section{T-duality for principal circle bundles in a background flux}\label{sect:T-duality}
In this section, we review the results in \cite{BEM,BEM2}, where the following situation is studied.
Let $Z$ be a principal $S^1$-bundle over $M$,
\begin{equation}\label{eqn:MVBx}
\begin{CD}
S^1 @>>> Z \\
&& @V\pi VV \\
&& M \end{CD}
\end{equation}
which is classified up to isomorphism by its first Chern class
{ $c_1(Z)\in H^2(M,\ZZ)$}. Assume that spacetime $Z$ is endowed with an H-flux which is
a representative in the
degree 3 Deligne cohomology of $Z$, that is
$H\in\Omega^3(Z)$ with integral periods (for simplicity, we drop factors of $\frac{1}{2\pi i}$),
together with
the following data. Consider a local trivialization $U_\alpha \times S^1$ of $Z\to M$, where
$\{U_\alpha\}$ is a good cover of $M$. Let $H_\alpha = H\Big|_{ U_\alpha \times S^1}
= d B_\alpha$, where $B_\alpha \in \Omega^2(U_\alpha \times S^1)$ and finally, $B_\alpha -B_\beta = F_{\alpha\beta}
\in \Omega^1(U_{\alpha\beta} \times S^1)$.
 Then the choice of H-flux entails that we are given a local trivialization
 as above and locally defined 2-forms $B_\alpha$ on it, together with closed 2-forms $F_{\alpha\beta}$ defined on double overlaps,  that is, $(H, B_\alpha, F_{\alpha\beta})$. This is also known as connection data or a connective structure of a gerbe on $Z$, see \cite{Brylinski}. Also the first Chern class
 of $Z\to M$ is represented in integral cohomology by the pair $(F, A_\alpha)$ where
$\{A_\alpha\}$ is a connection 1-form on the principal circle bundle 
$Z\to M$ and $F = dA_\alpha$ is the curvature 2-form of $\{A_\alpha\}$.

The {{T-dual}} is then another principal
$S^1$-bundle over $M$, denoted by $\widehat Z$,
  {}
\begin{equation}\label{eqn:MVBy}
\begin{CD}
S^1 @>>> \widehat Z \\
&& @V\widehat \pi VV     \\
&& M \end{CD}
\end{equation}
To define it, recall that $\pi_*$ denotes integration in the circle direction. Then 
we see that $\pi_* (H_\alpha) = d \pi_*(B_\alpha) = d {\widehat A}_\alpha$,
 so that $\{{\widehat A}_\alpha = \pi_*(B_\alpha)  \}$ is a connection 1-form whose curvature $ d {\widehat A}_\alpha = \widehat F_\alpha =  \pi_*(H_\alpha)$
 that is, $\widehat F = \pi_* H$. So let $\widehat Z$ denote the principal
$S^1$-bundle over $M$ whose first Chern class is  $\,\, c_1(\widehat Z) = [\pi_* H, \pi_*(B_\alpha)] \in H^2(M, \ZZ) $.

The Gysin sequence for $Z$ enables us to define a T-dual H-flux
$[\widehat H]\in H^3(\widehat Z,\ZZ)$, satisfying
\begin{equation} \label{eqn:MVBc}
c_1(Z) = \widehat \pi_* \widehat H \,,
\end{equation}
where $\pi_* $ and similarly $\widehat\pi_*$, denote the pushforward maps.
Note that $ \widehat H$ is not fixed by this data, since any integer
degree 3 cohomology class on $M$ that is pulled back to $\widehat Z$
also satisfies the requirements. However, $ \widehat H$ is
determined uniquely (up to cohomology) upon imposing
the condition $[H]=[\widehat H]$ on the correspondence space $Z\times_M \widehat Z$
as will be explained now.

The {\em correspondence space} (sometimes called the {\em doubled space}) is the fibred product 
and is defined as
$$
Z\times_M  \widehat Z = \{(x, \widehat x) \in Z \times \widehat Z|\ \pi(x)=\widehat\pi(\widehat x)\}.
$$
Then we have the following commutative diagram,
\begin{equation*} \label{eqn:correspondence}
\xymatrix @=6pc @ur { (Z, H) \ar[d]_{\pi} &
(Z\times_M  \widehat Z, [H]=[\widehat H]) \ar[d]_{\widehat p} \ar[l]^{p} \\ M & (\widehat Z, \widehat H)\ar[l]^{\widehat \pi}}
\end{equation*}
By requiring that
$$
[H]=[\widehat H] \in H^3(Z\times_M  \widehat Z, \ZZ),
$$
determines $[\widehat H] \in H^3(\widehat Z, \ZZ)$  uniquely, via an application of the Gysin sequence.
An alternate way to see this is explained below.

Let $(H, B_\alpha, F_{\alpha\beta}, L_{\alpha\beta})$ denote a gerbe with connection on $Z$.
We also choose a connection 1-form $A$ on $Z$. As usual, $X_A$ will denote the vector field generating the $S^1$-action on $Z$, normalized so that $\iota_A A = 1$, where $\iota_A$ is the contraction with respect to $X_A$. Then define $\widehat A_\alpha = -\imath_A B_\alpha$ on the chart $U_\alpha$ and
the connection 1-form $\widehat A= \widehat A_\alpha +d\widehat\theta_\alpha$
on the chart $U_\alpha\times S^1$. In this way we get a T-dual circle bundle
$\widehat Z \to M$ with connection 1-form $\widehat A$.

Without loss of generality, we can assume that $H$ is $S^1$-invariant. Consider
$$
\Omega = H - A\wedge F_{\widehat A}
$$
where $F_{\widehat A} = d {\widehat A}$ and $F_{A} = d {A}$ are the curvatures of $A$
and $\widehat A$ respectively. One checks that the contraction $\iota_A(\Omega)=0$ and
the Lie derivative $\text{Lie}_A(\Omega)=0$ so that $\Omega$ is a basic 3-form on $Z$, that is
$\Omega$ comes from the base $M$.

Setting
$$
\widehat H = F_A\wedge {\widehat A} + \Omega
$$
this defines the T-dual flux 3-form. One verifies that $\widehat H$ is a closed 3-form on $\widehat Z$.
It follows that on the correspondence space, one has as desired,
\begin{equation}
\widehat H = H + d (A\wedge \widehat A ).
\end{equation}

Our next goal is to determine the T-dual curving or B-field.
The Buscher rules imply that on the open sets $U_\alpha \times S^1\times S^1$ of the
correspondence space $Z\times_M \widehat Z$, one has
\begin{equation}
\widehat B_\alpha = B_\alpha + A\wedge \widehat A - d\theta_\alpha \wedge d\widehat \theta _\alpha\,,
\end{equation}
Note that
\begin{equation}
\imath_A \widehat B_\alpha = \imath_A
\left( B_\alpha + A\wedge \widehat A - d\theta_\alpha \wedge d\widehat \theta _\alpha\right) =
-\widehat A_\alpha + \widehat A - d\widehat \theta_\alpha = 0
\end{equation}
so that $\widehat B_\alpha$ is indeed a 2-form on $\widehat Z$ and not just on the correspondence
space. Obviously, $d \widehat B_\alpha = \widehat H$. Following the descent equations one arrives at the complete
T-dual gerbe with connection, $(\widehat H, \widehat B_\alpha, \widehat F_{\alpha\beta}, \widehat L_{\alpha\beta})$.
cf. \cite{BMPR}

The Buscher rules \cite{Buscher} for transforming the Ramond-Ramond (RR) fields can be encoded in the {\cite{BEM,BEM2}} generalization of
{\em Hori's formula} (see Hori \cite{Hori})
  {}

\begin{equation} \label{eqn:Hori}
T: \Omega^{\bullet}(Z)^{S^1} \ra \Omega^{\bullet+1}(\widehat{Z})^{S^1},\qquad T(G) =  \int_{S^1} e^{ A \wedge \widehat A }\ G \,,
\end{equation} where $G \in \Omega^\bullet(Z)^{S^1}$ is the total RR fieldstrength,
\begin{center}
$G\in\Omega^{\text{even}}(Z)^{S^1} \quad$ for {   { Type IIA}};\\
$G\in\Omega^{\text{odd}}(Z)^{S^1} \quad$ for {   { Type IIB}},\\
\end{center}
and where the right hand side of equation \eqref{eqn:Hori} is an invariant differential form on $Z\times_M\widehat Z$, and
the integration is along the $S^1$-fibre of $Z$.

We may write $G \in \Omega^{\bullet}(Z)^{S^1}$ in the form $G = G_0 + A\wedge G_1$, where $G_0 \in \pi^*(\Omega^{\bullet}(M))$ and $G_1 \in \pi^*(\Omega^{\bullet+1}(M))$. It is not difficult to check that \begin{equation} \label{dimredforms}T(G) = -G_1 + \widehat{A} \wedge G_0.\end{equation}

Recall that the twisted cohomology
is defined as the cohomology of the complex
$$H^\bullet(Z, H) = H^\bullet(\Omega(Z), d_H=d+ H\wedge).$$
By \eqref{dimredforms}, $T$ intertwines the differentials $d_H$ and $d_{\widehat H}$,
in particular taking $d_H$-closed forms $G$ to $d_{\widehat H}$-closed forms $T(G)$
and $d_H$-exact forms to $d_{\widehat H}$-exact forms. So $T$  induces a map on twisted cohomologies,
$$
T : H^\bullet(Z, H) \to H^{\bullet +1}(\widehat Z, \widehat H).
$$
Define the Riemannian metrics on $Z$ and $\widehat Z$ respectively by
$$
g=\pi^*g_M+R^2\, A\odot A,\qquad \widehat g=\widehat\pi^*g_M+1/{R^2} \,\widehat A\odot\widehat A.
$$
where $g_M$ is a Riemannian metric on $M$. Then $g$ is $S^1$-invariant and the length of each circle fibre is $R$; $\widehat g$ is $S^1$-invariant and the length of each circle fibre is $1/R$.

The following theorem summarizes the main consequence of T-duality for principal circle bundles in a background flux.

\begin{theorem}[T-duality isomorphism \cite{BEM,BEM2}]\label{thm:T-duality}
In the notation above, and
with the above choices of Riemannian metrics and flux forms, the map \eqref{eqn:Hori} is an isometry, inducing an isomorphism on twisted
cohomology groups,
\begin{equation}\label{T-duality-coh}
T : H^\bullet(Z, H) \stackrel{\cong}{\longrightarrow} H^{\bullet +1}(\widehat Z, \widehat H).
\end{equation}
Therefore under T-duality one has the exchange,
\begin{center}
{$R \Longleftrightarrow 1/R$}\quad and \quad
{{   {background H-flux} $\Longleftrightarrow$   {Chern class}}}
\end{center}
Moreover there is also an isomorphism of twisted K-theories,
\begin{equation}\label{T-duality-K}
T : K^\bullet(Z, H) \to K^{\bullet +1}(\widehat Z, \widehat H),
\end{equation}
such that the following diagram commutes,
\begin{equation}\label{T-duality-commute}
\xymatrix @=4pc 
{ K^\bullet(Z, H) \ar[d]_{Ch_H}  \ar[r]^{T} &
K^{\bullet +1}(\widehat Z, \widehat H)  \ar[d]^{Ch_{\widehat H}}   \\ H^\bullet(Z, H)  \ar[r]_{T}& H^{\bullet +1}(\widehat Z, \widehat H)}
\end{equation}

\end{theorem}
The surprising {\em {{new}}} phenomenon discovered in \cite{BEM,BEM2} is that there is a
{\em {change in topology}} when the H-flux is non-trivial.

\subsection{T-duality of Courant algebroids}
As shown by Cavalcanti and Gualtieri in \cite{cavalcanti} (see also \cite{Baraglia}), the isomorphism $
T\colon\Omega^{\bullet}(Z)^{S^1} \to\Omega^{\bullet +1}(\widehat Z)^{S^1}$ given by \eqref{eqn:Hori} is compatible with an isomorphism of Courant algebroids $$\tau: (TZ \oplus T^*Z,[\cdot ,\cdot]_H)/{S^1} \ra (T\widehat{Z} \oplus T^*\widehat{Z},[\cdot ,\cdot]_{\widehat{H}})/{S^1}.$$


Recall that in the dimension reduction formalism, we identify $\text{Vect}_{\text{hor}}(Z)$ with $\text{Vect}(M)$. We write a vector field in $TZ/S^1$ as $$(X,f) = X + f X_A,\qquad X \in   \text{Vect}(M),\qquad f \in C^{\infty}(M).$$ Similarly, we write a form in $(T^*Z)/S^1$ as $$(\omega, g) = \omega + g A,\qquad \omega \in \Omega^1(M),\qquad g \in C^{\infty}(M).$$

\begin{theorem} \label{thm:CG} (Cavalcanti, Gualtieri \cite{cavalcanti}). Given a T-dual pair $(Z,H)$, $(\widehat{Z},\widehat{H})$, the map $\tau: (TZ \oplus T^*Z,[\cdot ,\cdot]_H)/{S^1} \ra (T\widehat{Z} \oplus T^*\widehat{Z},[\cdot ,\cdot]_{\widehat{H}})/{S^1}$ given by \begin{equation} \label{defoftau} \tau \big((X,f) + (\omega,g)\big) = (X,g) + (\omega,f),\end{equation} is an isomorphism of Courant algebroids. Moreover, $\Omega^{\bullet}(Z)^{S^1}$ is a Clifford module over $(TZ\oplus T^*Z, [\cdot, \cdot]_H)/S^1$ via \eqref{cliffordmod}, and the maps \eqref{defoftau} and \eqref{eqn:Hori} are compatible in the obvious sense.  \end{theorem}


\section{Chiral T-duality}\label{sect:chiralT-duality}

As in the previous section, let $(Z,H)$ and $(\widehat{Z}, \widehat{H})$ be a T-dual pair of principal $S^1$-bundles on $M$. Our first observation is that if we combine Theorems \ref{thm:vanishing} and \ref{thm:T-duality}, we obtain the following result.
 \begin{theorem}[T-duality isomorphism of twisted chiral de Rham cohomology]\label{thm:chiral-T-duality I}
There is an isomorphism in twisted chiral de Rham cohomology 
\begin{equation}\label{chiral T-duality-coh}
T^{\text{ch}} :H^\bullet(\Omega^{ch}(Z), D_H) \stackrel{\cong}{\longrightarrow} H^{\bullet+1}(\Omega^{ch}(\widehat Z), D_{\widehat H}),\end{equation}
such that the following diagram commutes,
\begin{equation}\label{T-duality-commute}
\xymatrix @=4pc 
{ H^\bullet(\Omega^{\text{ch}}(Z), D_H)  \ar[d]_{\cong}  \ar[r]^{T^{\text{ch}}} &
H^{\bullet+1}(\Omega^{\text{ch}}(\widehat Z), D_{\widehat H}) \ar[d]^{\cong}   \\ H^\bullet(Z, H)  \ar[r]_{T}& H^{\bullet +1}(\widehat Z, \widehat H)}
\end{equation}
\end{theorem}

Next, we wish to find appropriate chiral analogues of the Cavalcanti-Gualtieri isomorphism of Courant algebroids \eqref{defoftau}, as well as the compatible degree-shifting isomorphism \eqref{eqn:Hori} of twisted de Rham complexes. As usual, we fix connection forms $A \in \Omega^1(Z)$ and $\widehat{A}\in \Omega^1(\widehat{Z})$, respectively, so that $\iota_{A} A = 1 = \iota_{\widehat{A}} \widehat{A}$, where $\iota_A$ and $\iota_{\widehat{A}}$ are the contractions with respect to the vector fields $X_A$ and $X_{\widehat{A}}$ generated by the $S^1$-actions. Fix an open cover $\{W_{\alpha}\}$ for $M$ which trivializes both $S^1$-bundles, such that each $W_{\alpha}$ is a coordinate open set. Then $\{V_{\alpha} = \pi^{-1}(W_{\alpha})\}$ and $\widehat{V}_{\alpha} = \widehat{\pi}^{-1}(W_{\alpha})$ are small open covers for $Z$ and $\widehat{Z}$, respectively, and $\{U_{\alpha} = p^{-1}(V_{\alpha}) = \widehat{p}^{-1}(\widehat{V}_{\alpha})\}$ is a small open cover for $Z\times_M \widehat{Z}$. Note that $H$ can be written in the form $H=H^{(3)} + A \wedge H^{(2)}$ where $H^{(3)} \in \pi^* (\Omega^3(M))$ and $H^{(2)}\in \pi^*(\Omega^2(M))$. Similarly, $\widehat{H} = \widehat{H}^{(3)} + \widehat{A} \wedge \widehat{H}^{(2)}$ where $\tilde{H}^{(3)} \in \widehat{\pi}^*(\Omega^3(M))$ and $\tilde{H}^{(2)} \in \widehat{\pi}^*(\Omega^2(M))$. We may assume without loss of generality that $$H^{(3)} = \widehat{H}^{(3)},\qquad H^{(2)} =  d \widehat{A} = F_{\widehat{A}},\qquad \widehat{H}^{(2)} = dA = F_A.$$

For simplicity of notation, for the remainder of this section we shall drop the tilde from our notation for the generators of both $\Omega^{\text{ch},H}(V_{\alpha})$ and $\Omega^{\text{ch},\widehat{H}}(\widehat{V}_{\alpha})$. First, we write down some OPE relations in $\Omega^{\text{ch},H}(V_{\alpha})^{i\mathbb{R}[t]}$ using the formula \eqref{dimredbracket} for the bracket $[\cdot, \cdot]_H$ in the dimension reduction formalism. We obtain
$$L_X(z) \iota_Y(w) \sim \big(\iota_{[X,Y]} + \iota_X \iota_Y H + :(\iota_X \iota_Y \widehat{H}^{(2)}) \iota_A:\big)(z-w)^{-1},$$
$$L_X(z) L_Y(w) \sim \big(L_{[X,Y]}+ D \iota_X \iota_Y H + :(D\iota_X \iota_Y \widehat{H}^{(2)}) \iota_A: + : L_A (\iota_X \iota_Y \widehat{H}^{(2)}):\big)(w)(z-w)^{-1}.$$
We have $\iota_A H = H^{(2)}$, and the quotient $\Omega^{\text{ch},H}(V_{\alpha})^{i\mathbb{R}[t]} / \langle L_A - H^{(2)} \rangle$ has strong generators $$f \in \pi^*(C^{\infty}(W_{\alpha})),\qquad \omega \in \pi^*(\Omega^1(W_{\alpha})),\qquad L_X, \iota_X|\ X\in \text{Vect}_{\text{hor}}(V_{\alpha}),\qquad \iota_A, A.$$ The OPE relations are the same except that $L_A$ is replaced with $H^{(2)}$, so that $$L_X(z) L_Y(w) \sim \big(L_{[X,Y]} + D \iota_X \iota_Y H + :(D \iota_X \iota_Y \widehat{H}^{(2)}) \iota_A: + :H^{(2)} (\iota_X \iota_Y \widehat{H}^{(2)}):\big)(w)(z-w)^{-1}.$$

\begin{theorem}[T-duality of twisted vertex algebra sheaves] 
For each $V_{\alpha}$, the map \begin{equation} \label{tauchiral} \tau^{\text{ch}}: \Omega^{\text{ch},H}(V_{\alpha})^{i\mathbb{R}[t]} / \langle L_A - H^{(2)}\rangle  \ra  \Omega^{\text{ch},\widehat{H}}(\widehat{V}_{\alpha})^{i\mathbb{R}[t]} / \langle L_{\widehat{A}} - \widehat{H}^{(2)}\rangle  \end{equation} defined on generators by 
\begin{equation} 
\begin{split}
A \mapsto \iota_{\widehat{A}},\qquad \iota_A \mapsto  \widehat{A}, \qquad \iota_X \mapsto \iota_X, \qquad L_X \mapsto L_X ,\qquad X\in \text{Vect}_{\text{hor}}(V_{\alpha}),\\
f \mapsto f,\qquad \omega \mapsto \omega,\qquad f\in \pi^*(C^{\infty}(W_{\alpha})),\qquad \omega \in \pi^*(\Omega^1(W_{\alpha})),
\end{split}
\end{equation}
is a vertex algebra isomorphism. Since $\tau^{\text{ch}}$ is defined in a coordinate-independent way, it is well-defined on overlaps. We obtain an isomorphism of vertex algebra sheaves on $M$,
\begin{equation} \label{tauchiralsheaf}\tau^{\text{ch}}: \pi_* \big(\big(\Omega^{\text{ch},H}_Z \big)^{i\mathbb{R}[t]} / \langle L_A - H^{(2)}\rangle\big) \ra 
\widehat{\pi}_* \big(\big(\Omega^{\text{ch},\widehat{H}}_{\widehat{Z}}\big)^{i\mathbb{R}[t]} / \langle L_{\widehat{A}}- \widehat{H}^{(2)}\rangle\big),\end{equation} which we also denote by $\tau^{\text{ch}}$. In particular, we have a vertex algebra isomorphism 
\begin{equation} \label{tauchiralglobal}\tau^{\text{ch}}: \Omega^{\text{ch},H}(Z)^{i\mathbb{R}[t]}/ \langle L_A - H^{(2)}\rangle \ra  \Omega^{\text{ch},\widehat{H}}(\widehat{Z})^{i\mathbb{R}[t]}/\langle L_{\widehat{A}}- \widehat{H}^{(2)}\rangle .\end{equation} \end{theorem}

\begin{proof} First, we check that $\tau^{\text{ch}}$ preserves OPEs.
$$L_X(z) \iota_A(w) \sim (\iota_X \iota_A H)(w)(z-w)^{-1} = (\iota_X H^{(2)})(w)(z-w)^{-1},$$ and
$$\tau^{\text{ch}}(L_X)(z) \tau^{\text{ch}}(\iota_A)(w)= L_X (z) \widehat{A}(w) \sim ( \iota_X H^{(2)}) (w)(z-w)^{-1}.$$ Here we are using $\iota_X \circ_0  \widehat{A} = 0$, so $L_X  \circ_0 \widehat{A} = \iota_X \circ_0 d \widehat{A} = \iota_X \circ_0 H^{(2)}$. The proof that $\tau^{\text{ch}}$ preserves the OPE of $L_X(z) A(w)$ is the same. Next, $$L_X(z) \iota_Y(w) \sim \big(\iota_{[X,Y]} + : \iota_A(\iota_X \iota_Y \widehat{H}^{(2)}): + \iota_X \iota_Y H\big)(w)(z-w)^{-1}$$ 
$$ = \big(\iota_{[X,Y]} +  : \iota_A(\iota_X \iota_Y \widehat{H}^{(2)}): + \iota_X \iota_Y H^{(3)} + :A (\iota_X \iota_Y H^{(2)}):\big)(w)(z-w)^{-1},$$ and 
$$\tau^{\text{ch}}\big(\iota_{[X,Y]} +  : \iota_A(\iota_X \iota_Y \widehat{H}^{(2)}): + \iota_X \iota_Y H^{(3)} + :A (\iota_X \iota_Y H^{(2)}):\big)$$ $$ = \iota_{[X,Y]} +: \widehat{A}(\iota_X \iota_Y \widehat{H}^{(2)}):+ \iota_X \iota_Y H^{(3)} + :\iota_{\widehat{A}} (\iota_X \iota_Y H^{(2)}):.$$ 
Moreover, 
$$\tau^{\text{ch}}(L_X)(z) \tau^{\text{ch}}(\iota_Y)(w) \sim \big(\iota_{[X,Y]} +  : \iota_{\widehat{A}} (\iota_X \iota_Y H^{(2)}): + \iota_X \iota_Y H^{(3)} $$ $$+ :\widehat{A} (\iota_X \iota_Y \widehat{H}^{(2)}):\big)(w)(z-w)^{-1}.$$ 
Next, $$L_X(z) L_Y(w) \sim  \big(L_{[X,Y]} + D(: \iota_A(\iota_X \iota_Y \widehat{H}^{(2)}):) + D\iota_X \iota_Y H^{(3)} $$ $$+ D(: A \iota_X \iota_Y H^{(2)}:)\big)(w)(z-w)^{-1},$$ whose first-order pole coincides with 
$$ L_{[X,Y]} + : L_A (\iota_X \iota_Y \widehat{H}^{(2)}): - :\iota_A(D \iota_X \iota_Y \widehat{H}^{(2)}): + D\iota_X \iota_Y H^{(3)} $$ $$ + :\widehat{H}^{(2)} ( \iota_X \iota_Y H^{(2)}): -  : A (D \iota_X \iota_Y H^{(2)}):.$$
We have $$\tau^{\text{ch}}\big(L_{[X,Y]} + : L_A(\iota_X \iota_Y \widehat{H}^{(2)}): - :\iota_A(D \iota_X \iota_Y \widehat{H}^{(2)}): + D\iota_X \iota_Y H^{(3)} $$ $$ + :\widehat{H}^{(2)} ( \iota_X \iota_Y H^{(2)}): -  : A (D \iota_X \iota_Y H^{(2)}:)\big) $$
$$ = L_{[X,Y]} + : H^{(2)}(\iota_X \iota_Y \widehat{H}^{(2)}): - :\widehat{A}(D \iota_X \iota_Y \widehat{H}^{(2)}): + D\iota_X \iota_Y H^{(3)} $$ $$+ :\widehat{H}^{(2)} ( \iota_X \iota_Y H^{(2)}): -  : \iota_{\widehat{A}} (D \iota_X \iota_Y H^{(2)}:).$$
Moreover,
$$\tau^{\text{ch}}(L_X)(z) \tau^{\text{ch}}(L_Y)(w) \sim  \big(L_{[X,Y]} + D(: \iota_{\widehat{A}} (\iota_X \iota_Y H^{(2)}):) $$ $$+ D\iota_X \iota_Y H^{(3)} + D(: \widehat{A} \iota_X \iota_Y \widehat{H}^{(2)}:)\big)(w)(z-w)^{-1},$$ and the first-order pole is easily seen to coincide with 
$$L_{[X,Y]} + : H^{(2)}(\iota_X \iota_Y \widehat{H}^{(2)}): - :\widehat{A}(D \iota_X \iota_Y \widehat{H}^{(2)}): + D\iota_X \iota_Y H^{(3)} $$ $$+ :\widehat{H}^{(2)} ( \iota_X \iota_Y H^{(2)}): -  : \iota_{\widehat{A}} (D \iota_X \iota_Y H^{(2)}):.$$
The remaining OPE relations are easy to check and are omitted. Finally, the fact that $\tau^{\text{ch}}$ annihilates the relations \eqref{cirelations} is clear, as well as the bijectivity. \end{proof}

We regard \eqref{tauchiralglobal} as the analogue of the Cavalcanti-Gualtieri isomorphism \eqref{defoftau}. Note that $\tau^{\text{ch}}$ preserves degree but not weight since $\iota_A, \iota_{\widehat{A}}$ have weight one, but $A, \widehat{A}$ have weight zero. 

\begin{remark} In general, $\tau^{\text{ch}}$ does {\it not} lift to an isomorphism of the vertex algebra sheaves \begin{equation} \label{liftediso} \tau^{\text{ch}}: \pi_* \big(\big(\Omega^{\text{ch},H}_Z \big)^{i\mathbb{R}[t]}\big) \ra \widehat{\pi}_* \big(\big(\Omega^{\text{ch},\widehat{H}}_{\widehat{Z}}\big)^{i\mathbb{R}[t]}\big).\end{equation} However, if $Z$ is {\it self-dual} in the sense that $c_1(Z) = c_1(\widehat{Z})$, we may choose $H, \widehat{H}$ so that $H^{(2)} = \widehat{H}^{(2)}$. Defining $\tau^{\text{ch}}(L_A) = L_{\widehat{A}}$, it is easy to check that $\tau^{\text{ch}}$ lifts to an isomorphism \eqref{liftediso}. In fact, defining $\tau^{\text{ch}}(\Gamma^A) = \Gamma^{\widehat{A}}$, one can check that \eqref{liftediso} actually extends to an isomorphism $$\tau^{\text{ch}}: \pi_* \big(\big(\Omega^{\text{ch},H}_Z \big)^{i\mathbb{R}}\big) \ra \widehat{\pi}_* \big(\big(\Omega^{\text{ch},\widehat{H}}_{\widehat{Z}}\big)^{i\mathbb{R}}\big).$$ The reason is that $\Gamma^A$ commutes with $f, \omega, \iota_X$, and satisfies \eqref{copyofheis} as well as
$$L_X(z) \Gamma^A(w) \sim \iota_X \xi^A(w)(z-w)^{-1}.$$ Recall that $\xi^A \in \pi^*(\Omega^{\text{ch}}_0(M))$ has weight one and degree one, and satisfies $D(\xi^A) = \partial dA$. Since $H^{(2)} = \widehat{H}^{(2)}$, we may also assume that $\xi^A = \xi^{\widehat{A}}$. \end{remark}

\begin{remark}Note that $\tau^{\text{ch}}$ does not intertwine the differentials, since $$\tau^{\text{ch}} \circ D (\iota_A) = 0,\qquad D\circ \tau^{\text{ch}}(\iota_A) = D(\widehat{A}) = H^{(2)}.$$ (Recall that we are omitting the tilde, so $D$ really means $\tilde{D}$). In general, $$H^*((\Omega^{\text{ch},H}(Z)^{i\mathbb{R}[t]} / \langle L_A -H^{(2)} \rangle), D) \ncong H^*((\Omega^{\text{ch}}(\widehat{Z})^{i\mathbb{R}[t]} / \langle L_{\widehat{A}} - \widehat{H}^{(2)} \rangle), D).$$ However, there are {\it modified} differentials $D+ H_0$ and $D + \widehat{H}_0$ on $\Omega^{\text{ch},H}(Z)^{i\mathbb{R}[t]} / \langle L_A -H^{(2)} \rangle$ and $\Omega^{\text{ch}}(\widehat{Z})^{i\mathbb{R}[t]} / \langle L_{\widehat{A}} - \widehat{H}^{(2)} \rangle$, respectively, such that $\tau^{\text{ch}} \circ (D + H_0) = (D + \widehat{H}_0) \circ \tau^{\text{ch}}$. Here $H_0$ and $\widehat{H}_0$ denote the zero modes of $H$ and $\widehat{H}$, regarded as vertex operators. These differentials are square-zero and homogeneous with respect to the $\mathbb{Z}/2\mathbb{Z}$ grading, and $$(D + H_0)(\iota_X) = L_X,\qquad (D + H_0)(L_X) = 0,\qquad (D + H_0)(f) = df,\qquad (D+ H_0)(\omega) = d\omega,$$ $$\tau^{\text{ch}} \circ (D+ H_0)(\iota_A) = \tau^{\text{ch}}(H^{(2)}) = H^{(2)} = (D+ \widehat{H}_0) \circ \tau^{\text{ch}} \iota_A  = (D+ \widehat{H}_0)(\widehat{A}).$$ It follows that we have an isomorphism of vertex algebras $$H^{\bullet} ((\Omega^{\text{ch},H, \bullet}(Z)^{i\mathbb{R}[t]} / \langle L_A -H^{(2)} \rangle), D + H_0) \cong H^{\bullet}((\Omega^{\text{ch}, \bullet}(\widehat{Z})^{i\mathbb{R}[t]} / \langle L_{\widehat{A}} - \widehat{H}^{(2)} \rangle), D + \widehat{H}_0).$$
 \end{remark}

Recall next that we have untwisting isomorphisms
$$(\Omega^{\text{ch}})^{i\mathbb{R}[t]}/ \langle L_A \rangle \rightarrow (\Omega^{\text{ch},H})^{i\mathbb{R}[t]} / \langle L_A -H^{(2)} \rangle,\qquad (\Omega^{\text{ch},\widehat{H}})^{i\mathbb{R}[t]} / \langle L_{\widehat{A}} \rangle \ra (\Omega^{\text{ch}})^{i\mathbb{R}[t]} / \langle L_{\widehat{A}} - \widehat{H}^{(2)} \rangle.$$ There is no analogue of this untwisting in the setting of Courant algebroids; it is a feature of the vertex algebra setting. \begin{corollary} Composing $\tau^{\text{ch}}$ with these isomorphisms, we obtain a degree-preserving isomorphism of untwisted vertex algebra sheaves on $M$, also denoted by $\tau^{\text{ch}}$,
\begin{equation}\tau^{\text{ch}}: \pi_* \big( \big(\Omega^{\text{ch}}_Z\big)^{i\mathbb{R}[t]} / \langle L_A \rangle \big) \rightarrow  \widehat{\pi}_* \big(\big(\Omega^{\text{ch}}_{\widehat{Z}}\big)^{i\mathbb{R}[t]} / \langle L_{\widehat{A}} \rangle \big),\end{equation} and the corresponding isomorphism of global section algebras $$ \tau^{\text{ch}}: \Omega^{\text{ch}}(Z)^{i\mathbb{R}[t]} / \langle L_A \rangle \rightarrow \Omega^{\text{ch}}(Z)^{i\mathbb{R}[t]} / \langle L_{\widehat{A}} \rangle .$$ \end{corollary}

To define the analogue of the map $T: \Omega^{\bullet}(Z)^{S^1} \ra \Omega^{\bullet +1}(\widehat{Z})^{S^1}$, we need to regard $\big(\Omega^{\text{ch}}_Z\big)^{i \mathbb{R}[t]}/ \langle L_A  \rangle $ not as a vertex algebra sheaf, but as a sheaf of modules over itself which we denote by $\big(\Omega^{\text{ch}, \bullet}_Z\big)^{i \mathbb{R}[t]}/ \langle L_A  \rangle$. For each $V_{\alpha}$, $(\Omega^{\text{ch},\bullet}(V_{\alpha}))^{i \mathbb{R}[t]} / \langle L_A  \rangle$ is generated by the vacuum vector $1$ as a module over itself.


\begin{theorem} For each $V_{\alpha}$, define a linear map 
\begin{equation} \label{tchiralmodule} T^{\text{ch}}: \Omega^{\text{ch}, \bullet}(V_{\alpha})^{i\mathbb{R}[t]} / \langle L_A \rangle \rightarrow  \Omega^{\text{ch}, \bullet +1}(\widehat{V}_{\alpha})^{i \mathbb{R}[t]} / \langle L_{\widehat{A}} \rangle,\end{equation} inductively as follows:
$$T^{\text{ch}}(1) = \widehat{A}, \qquad T^{\text{ch}}(\nu_k (\mu)) = (-1)^{|\nu|} (\tau^{\text{ch}}(\nu))_k (T^{\text{ch}}(\mu)).$$ Here $\nu \in \Omega^{\text{ch}}(V_{\alpha})^{i\mathbb{R}[t]} / \langle L_A \rangle$ regarded as a vertex algebra, and $\mu \in \Omega^{\text{ch},\bullet}(V_{\alpha})^{i\mathbb{R}[t]} / \langle L_A  \rangle$ regarded as a module. Then $T^{\text{ch}}$ is a linear isomorphism of modules which preserves weight and shifts degree. Moreover, $T^{\text{ch}}$ coincides at weight zero with the classical T-duality map \eqref{eqn:Hori}. We obtain a weight-preserving, degree-shifting isomorphism of sheaves of modules on $M$ 
\begin{equation} \label{T-duality-module}T^{\text{ch}}: \pi_*\big(\big(\Omega^{\text{ch}, \bullet}_Z\big)^{i\mathbb{R}[t]} / \langle L_A  \rangle \big) \rightarrow  \widehat{\pi}_*\big( \big(\Omega^{\text{ch}, \bullet +1}_{\widehat{Z}}\big)^{i\mathbb{R}[t]}/ \langle L_{\widehat{A}} \rangle\big),\end{equation}  which we also denote by $T^{\text{ch}}$. In particular, we get a linear isomorphism of global sections
\begin{equation} \label{T-duality-moduleglobal}T^{\text{ch}}: \Omega^{\text{ch}, \bullet}(Z)^{i\mathbb{R}[t]} / \langle L_A  \rangle \rightarrow  \Omega^{\text{ch}, \bullet +1}(\widehat{Z})^{i\mathbb{R}[t]} / \langle L_{\widehat{A}} \rangle.\end{equation} \end{theorem}

\begin{proof} The fact that $T^{\text{ch}}$ is well-defined is a consequence of the standard quasi-commutativity and quasi-associativity formulas in vertex algebra theory. Note that $$T^{\text{ch}}(A) = T^{\text{ch}}(A_0 (1)) = -(\tau^{\text{ch}}(A))_0 (T^{\text{ch}}(1)) = -(\iota_{\widehat{A}})_0(\widehat{A}) = - 1.$$ Since $\tau^{\text{ch}}(\omega) = \omega$ for all $\omega \in \pi^*(\Omega^1(W_{\alpha}))$, it follows that at weight zero, $T^{\text{ch}}$ coincides with \eqref{eqn:Hori}. \end{proof}

At weight zero, \eqref{T-duality-moduleglobal} induces the classical T-duality isomorphism \eqref{T-duality-coh} in twisted cohomology. However, this map does {\it not} intertwine the differentials $D_H$ and $D_{\widehat{H}}$ in positive weight. If $Z$ has finite topological type, by combining the isomorphism \eqref{T-duality-coh} with the description $H^{\bullet}((\Omega^{\text{ch}}(Z)^{i \mathbb{R}[t]} / \langle L_A \rangle), D_H) \cong H^{\bullet}(Z,H) \otimes \cF$ given by Corollary \ref{twistedcohomology-quot}, we obtain a linear isomorphism 
\begin{equation} \label{T-duality-higherweightcoh}H^{\bullet}((\Omega^{\text{ch}}(Z)^{i \mathbb{R}[t]} / \langle L_A \rangle), D_H) \cong H^{\bullet +1 }((\Omega^{\text{ch}}(\widehat{Z})^{i \mathbb{R}[t]} / \langle L_{\widehat{A}}\rangle), D_{\widehat{H}}).\end{equation} In particular, the graded characters of these structures agree.




\end{document}